\documentclass{amsart}

\usepackage{amssymb,latexsym,ulem,array,color}

\newcommand\redST{\bgroup\markoverwith{\textcolor{red}{\rule[0.5ex]{2pt}{0.4pt}}}\ULon}

\newtheorem{lemma}{Lemma}[section]
\newtheorem{theorem}[lemma]{Theorem}
\newtheorem{corollary}[lemma]{Corollary}
\newtheorem{proposition}[lemma]{Proposition}
\newtheorem{definition}[lemma]{Definition}
\theoremstyle{remark}
\newtheorem*{acknowledgements}{Acknowledgements}
\newtheorem{remark}[lemma]{Remark}
\newtheorem{example}[lemma]{Example}

\newcommand\AAA{\mathbb A}
\newcommand\BB{\mathbb B}
\newcommand\CC{\mathbb C}

\newcommand\PP{\mathbb P}
\newcommand\ZZ{\mathbb Z}

\newcommand\sA{\mathcal A}
\newcommand\sB{\mathcal B}
\newcommand\sC{\mathcal C}
\newcommand\sD{\mathcal D}
\newcommand\sE{\mathcal E}
\newcommand\sF{\mathcal F}

\newcommand\sI{\mathcal I}

\newcommand\sM{\mathcal M}
\newcommand\sN{\mathcal N}
\newcommand\sO{\mathcal O}
\newcommand\sP{\mathcal P}

\newcommand\sR{\mathcal R}

\newcommand\sT{\mathcal T}

\newcommand\sV{\mathcal V}

\newcommand\hol{\mathcal O}
\newcommand\ra{\rightarrow}

\newcommand{\Proj}{\operatorname{Proj}}
\newcommand{\Aut}{\operatorname{Aut}}

\numberwithin{equation}{section}

\title[Even surfaces of general type with $K^2=8$, $p_g=4$ and 
$q=0$]{The moduli space of even surfaces of general type  with 
$K^2=8$, $p_g=4$ and
$q=0$}

\author{Fabrizio Catanese}
\address{Fabrizio Catanese \\Lehrstuhl Mathematik 
VIII\\Mathematisches Institut der Universit\"at Bayreuth\\ 
NWII\\Universit\"atstr. 30\\95447
Bayreuth\\Germany}
\email{fabrizio.catanese@uni-bayreuth.de}

\author{Wenfei Liu}
\address{Wenfei Liu \\Fakult\"at f\"ur Mathematik\\Universit\"at 
Bielefeld\\Universit\"atsstr. 25\\33615 Bielefeld\\Germany}
\email{liuwenfei@math.uni-bielefeld.de}

\author{Roberto Pignatelli}
\address{Roberto Pignatelli \\Dipartimento di 
Matematica\\Universit\`a di Trento\\via Sommarive 14\\38123 
Trento\\Italy}
\email{roberto.pignatelli@unitn.it}

\begin{document}

\thanks{The present cooperation  was supported by the DFG
Forschergruppe 790 `Classification of algebraic surfaces and
compact complex manifolds', and by the Emmy Noether  Nachwuchsgruppe
`Modulr\"aume und Klassifikation von algebraischen Fl\"achen und Nilmannigfaltigkeiten mit linksinvarianter komplexer Struktur'}

\date{\today}

\begin{abstract}

  Even surfaces of general type with
$K^2=8$, $p_g=4$ and $q=0$ were found by Oliverio \cite{O} as 
complete intersections of bidegree $(6,6)$ in a weighted projective 
space
$\PP(1,1,2,3,3)$.

In this article we prove
that the moduli space of even surfaces of general type with
$K^2=8$, $p_g=4$ and $q=0$ consists of two
$35$-dimensional irreducible components intersecting in a codimension 
one subset (the first  of these components is the closure of the open set
considered by  Oliverio).
 All the surfaces in  the second  component have a singular  canonical model,
hence we get a new example of a generically nonreduced moduli space.

Our result  gives  a posteriori a complete description of the half-canonical rings of the above even surfaces.
The method of proof is, we believe,  the most interesting part of the paper. After describing the graded  ring
of a cone we are able, combining the explicit description of some subsets of the moduli space, some deformation
theoretic arguments, and finally some local algebra arguments, to describe the whole moduli space.

This is the first time that the classification of a class of surfaces can only be done  using
moduli theory: up to now  first the surfaces were classified, on the basis of some numerical inequalities,
or other arguments, and later on the moduli spaces were investigated.

\end{abstract}
\maketitle

\section{Introduction} Algebraic surfaces with geometric genus 
$p_g=4$ have been a very natural object of study since
  Noether's seminal paper \cite{noether} in the 19-th century. Because 
their canonical map  has  image
  $\Sigma_1$ which most of the times  is a surface in the projective 3-dimensional space $\PP^3$, 
defined therefore by a single polynomial equation.

In the  20-th century new examples of such surfaces were found by 
several authors (\cite{maxwell}, \cite{franchetta}, \cite{enriques},
\cite{kodaira}), and  a substantial part of Chapter VIII of 
Enriques's book \cite{enriques} is devoted to the discussion and the 
proposal of
several constructions, in the range  $4\leq K^2\leq 10$. New examples 
were then found in  \cite{babbage} and \cite{C} (see also 
\cite{bucharest}).
Nowadays the investigation of such surfaces is an interesting chapter 
of the theory of
surfaces with small invariants, encompassing (easier)  existence 
questions and (harder) investigation of    moduli spaces.

By the inequalities of Noether and Bogomolov-Miyaoka-Yau, minimal 
surfaces of general type with $p_g=4$ satisfy $4\leq K^2\leq 45$.
Only recently the upper bound $ K^2 =  45$ was shown to be achieved 
(\cite{45}), while the first historical examples of surfaces which we 
mentioned above
are surfaces with $4\leq K^2\leq 7$; by the work of Ciliberto and 
Catanese, \cite{C} and \cite{sbc},  existence is known for each $4\leq 
K^2\leq 28$.

Irregular surfaces with $p_g=4$ were later investigated in \cite{cs}: 
in this case $K^2 \geq 8$ since, by \cite{debarre}, one has
$K^2 \geq  2 p_g$ for irregular surfaces  \footnote{ That the case 
$K^2=8$, $p_g=4$ and $q=1$ actually occurs is shown by
  the family of double covers of the product   $E \times \PP^1$, where 
$E$ is an elliptic curve and the branch divisor has numerical type 
$(4,6)$.};
  while $K^2 \geq
12$ if the canonical map has degree 1.

Surfaces with $p_g=4$ and $K^2 = 4$ were classified by Noether and Enriques, but 
it took the work of  Horikawa and
Bauer  ( \cite{Ho1,Ho2,Ho3,B}) to finish the classification of  the 
surfaces with $p_g=4$ and $4\leq K^2\leq 7$ (necessarily regular).
These  are `essentially' classified, in the sense that the moduli 
space is shown to be a union of certain (explicitly described) locally closed subsets:
  but there is   missing complete knowledge
of the incidence structure of these subsets of the  moduli space.  We 
refer to the survey \cite{BCP3} for a good account
of the range  $4\leq K^2\leq 7$, and to
\cite{homalg} for a previous more general survey (containing the 
construction of several new examples).

Minimal surfaces with $K^2=8,p_g=4,q=0$ have  been the object of 
further work  by several authors \cite{C,CFM,O}.
  The  surfaces constructed by Ciliberto have a birational canonical 
map,   are not even, and have a trivial torsion group  
$H_1(S, \ZZ)$
(unlike the ones considered in \cite{CFM}); the ones constructed by Oliverio  are simply
connected (see \cite{D}), and they are even (meaning that the 
canonical divisor is divisible by two: i.e.,
the second Stiefel Whitney class $w_2(S) = 0$, equivalently, the 
intersection form is even).

Therefore, for $K^2=8,p_g=4,q=0$ there are at least three
different topological types \cite[Remark 5.4]{O}, contrasting the 
situation for (minimal) surfaces with $p_g = 4$, $K^2\leq 7$ which,
when they have the same $K^2$, are homeomorphic to each
other. Recently Bauer and the third author \cite{BP} classified 
minimal surfaces with $K^2=8,p_g=4,q=0$ whose canonical map is 
composed with an
involution (while examples with canonical map of degree three are 
given in \cite{MP}).

 Their work shows that the moduli space of 
minimal surfaces with
$K^2=8,p_g=4,q=0$ has at least four irreducible components: and a new fifth one is
described in the present paper.

Therefore the classification of minimal surfaces with 
$K^2=8,p_g=4,q=0$ seems a very challenging problem, yet not 
completely out of reach.

The present article provides a first step in this direction, 
classifying all the even surfaces and completely describing the 
corresponding subset
$M^{ev}_{8,4,0}$ of the moduli space.

  This is our main result: denote by $M^{ev}_{8,4,0}$   the moduli 
space of even surfaces of general type with $K^2=8$, $p_g=4$ and 
$q=0$.  We show that
$M^{ev}_{8,4,0}$, which a priori  consists of several connected 
components of the whole moduli space $M_{8,4,0}$, is indeed a single 
connected component
  of the  moduli space $M_{8,4,0}$.  Oliverio
\cite{O} found out that, if  $|K_S|$ is base point free (this 
condition determines an  open set of the moduli space)  the 
half-canonical ring $R(S,L)$ realizes
the canonical model $X$ of the surface $X$ as a
$(6,6)$ complete  intersection in $\PP(1^2,2,3^2)$. Since conversely 
such complete intersections having at worst Du Val singularities
(i.e., rational double points) yield such canonical models, one gets 
as a result that this open set
is an irreducible unirational open set  of dimension
$35$ in the  moduli space $M^{ev}_{8,4,0}$,  hence it gives rise to 
an irreducible component which we denote by
$M_{\sF}$ (here $\sF$ stands for ``free" canonical system). We treat 
here the case where
$|K_S|$ has base points and, completing Oliverio's result, we obtain 
the following
\begin{theorem} The moduli space $M_{8,4,0}^{ev}$ consists of two 
$35$-dimensional irreducible components
$M_{\sF}$ and $M_{\sE}$, such that the general points of $M_{\sF}$ 
correspond to surfaces with base  point free canonical systems, while 
all points of
$M_{\sE}$ correspond to surfaces whose canonical  system has base 
points. Moreover $M_{\sF}$ and
$M_{\sE}$ intersect in a codimension  one
irreducible  subset.
\end{theorem}

The subscript $\sE$ in $M_{\sE}$ stands for `extrasymmetric', 
since the general points of $M_{\sE}$ correspond to
surfaces whose half-canonical ring admits an `extrasymmetric' presentation.

 Let us also point out that we are describing the Gieseker moduli space,
and that in fact all the surfaces in the component $M_{\sE}$ have a node,
hence the Kuranishi family for $S$ is non-reduced at each point of this component.

Our proof  starts with Reid's method of infinitesimal extension of 
hyperplane sections  (cf. \cite{R1}), which is the algebraic 
counterpart
(in terms of graded rings) of the inverse of the classical 
geometrical method of sweeping the cone: taking the projective cone
$Cone(X)$ with vertex $P$ over  a projective variety, any hyperplane 
section of $Cone(X)$ not passing through $P$ is isomorphic to
$X$, and one can make it degenerate to a hyperplane section of the 
cone passing through $P$, which  is the cone $Cone(H\cap X)$ with 
vertex $P$ over the hyperplane
section of
$X$,
$ H \cap X$. Viewing the process in the inverse way, one may see $X$ 
as a deformation of $Cone(H\cap X)$ and indeed Schlessinger, Mumford 
and Pinkham
(\cite{schlessinger, mumford, pinkham}) set up the theory of 
deformations of varieties  with a $\CC^*$-action to analyse this 
situation.

The advantage in the surface case is that the hyperplane  section is 
a curve $C$ and the graded ring of the cone over $C$ is much more 
tractable than in the higher dimensional case.

It is so once more in our special situation. In order to describe the graded 
ring  $R(S,L)$ associated to a half  canonical divisor $L$, we first 
calculate (Proposition~
\ref{rc2p}), in the case where the canonical system has base points, the quotient ring $R$ associated to the restriction  of 
$R(S,L)$ to a smooth curve $C$ in $|L|$.  Then  we would like  to recover $R(S,L)$ 
as an extension ring,
which of course can be viewed as a deformation of the cone  $C_R$ 
associated to the graded ring $R$.

We can find some of these extension rings using  two different ``formats", an old one and a new one;
the old one  consists in  writing the relations in
$R$ in terms of  pfaffians of certain extrasymmetric skewsymmetric matrices 
(Example \ref{ex: extrasym}), while the new one is more complicated (see  Example  \ref{ex: MV}). These formats 
produce in a natural way
two families of such surfaces embedded in a weighted projective space  of 
dimension 6 (see Propositions \ref{ME} and \ref{MV}) via their half-canonical rings.

As written in the abstract, our method gives only a posteriori a complete description 
of the half-canonical rings of the above surfaces (this was first achieved by the second author
via  heavy computer-aided calculations  which are impossible to be reproduced in a paper).

  We then prove that these two families fill the
moduli space via a crucial
 study of the local deformation space of the cone $C_R$, obtained by first studying the  infinitesimal deformations of 
first and second order.  Then, using some local 
algebra arguments, we 
show that there cannot be  higher order obstructions.

An interesting novel  feature is that the deformation space has  embedding codimension two and is
not a local complete intersection.

\begin{acknowledgements}
  Part of this paper is contained in the second author's thesis (in 
Chinese) submitted to Peking University in fulfillment
of the duties required for earning the doctoral degree; the second 
author   would like to
thank his advisors J.~Cai and F.~Catanese for their support and for 
useful suggestions during the research. The authors  would like to thank Ingrid Bauer and Paolo Oliverio
for useful conversations.
\end{acknowledgements}

\section{Preliminaries} This section collects some  notions and facts 
that will be used in the sequel.

\subsection{Notation} The varieties that we consider  in this paper are defined over 
the complex numbers $\CC$.

Throughout the article $S$ shall be the minimal model of a surface of 
general type with $K_S^2=8$, $p_g(S)=4$ and $q (S)=0$,
and we shall assume that $S$ is {\bf even}, which means that there 
exists a divisor class $L$ such that $ 2L \equiv K_ S $, $\equiv$ 
denoting, as classically, linear
equivalence. The first simple observation is that, if $S$ is even, 
then $S$ is minimal, since for an exceptional curve
of the first kind $E$ one would have $ -1 = E \cdot K_S = 2 E \cdot 
L$, a contradiction.

Observe moreover that $L$ is not a priori unique, since the class of 
$L$ is determined up to addition of a  $2$-torsion divisor class,
and these form a finite group (only a posteriori we shall see that 
the class of $L$ is unique, since  all our surfaces will be shown to
be deformation equivalent to a weighted  complete intersection, hence
they are all simply connected).

Hence we shall throughout consider a pair $(S,L)$ as above, observing 
that deformations of the pair $(S,L)$ correspond to deformations of 
$S$ up to an \'etale base
change.

  Given a projective algebraic variety $Y$ and a line bundle $L$,
$$R(Y,L):=\oplus_{n\geq
0}H^0(Y,nL)$$ is the graded  ring of sections of the pair $(X,L)$.

The canonical model of such a surface of general type $S$ is the 
projective spectrum $X : =  Proj ( \sR (S, K_S))$ of the canonical 
ring of $S$;
$X$ is also  the projective spectrum $=  Proj ( \sR (S, L))$ of the 
semicanonical ring associated to the class of $L$, and is the
only birational model of $S$ with ample canonical class and at most 
rational double points as singularities. The above (graded) rings are 
all
finitely generated $\CC$-algebras. Observe that the line bundle $L$ 
descends to a line bundle on $X$ by the results of Artin in 
 \cite{artin}.

Observe that any deformation of $S$ yields a family of relative 
canonical algebras and (up to \'etale base change) a family of 
relative half canonical algebras.
In particular, by the semicontinuity theorem, a  minimal system of (homogeneous)
generators for the semicanonical ring $ \sR (S, L)$ yields an 
embedding of $X$
into a weighted projective space $\PP(d_1, d_2, \dots , d_h)$, and 
locally a deformation of $S$ yields a family of projectively normal 
subschemes
of $\PP(d_1, d_2, \dots , d_h)$. 

By
abuse of notation, we denote sometimes an element of a polynomial 
ring and its image in some quotient ring by the same symbol. 

$\CC[\epsilon]:=\CC[t]/(t^2)$ is
the ring of dual numbers, so that $\epsilon$ is meant to be a first 
order parameter, i.e., $\epsilon$ has degree $0$ and $\epsilon^2=0$.
  Given a
$\CC$-algebra $R$, $R[\epsilon]:=R\otimes_\CC \CC[\epsilon]$.

  Finally, let $M=(m_{ij})$ be a $4 \times 4$ skewsymmetric matrix. 
Then recall that
the  pfaffian of $M$
is
  \[
  pf(M)=m_{12}m_{34}-m_{13}m_{24}+m_{14}m_{23}.
\]

\subsection{Deformation of closed subschemes and graded 
rings}\label{subschemedeformationsection}

As we shall sometimes not only use the analytical theory of 
deformations, we observe the connection between Hilbert schemes
and deformation theory, which is central in our arguments.

\begin{theorem}[Ideal-Variety Correspondence]
  There is a natural bijection between the set of closed subschemes of 
a weighted projective space and the set of saturated homogeneous 
ideals of the
(weighted) polynomial ring, which associates any closed subscheme to 
its saturated homogeneous ideal.
\end{theorem}
\begin{proof} For closed subschemes of usual projective space, see 
for example \cite[Ex.~II.5.10]{Ha1}; for the weighted case, see
\cite[3.1.2(iv)]{D}.
\end{proof}

\begin{theorem}\label{subschemedeformation}
  The deformation functor of closed subschemes in weighted projective 
space satisfies Schlessinger's conditions 
$H_0,H_{\epsilon},\bar{H},H_f$, and
hence is prorepresentable.
\end{theorem}
\begin{proof} This is \cite[Prop.~3.2.1]{Sern}.
\end{proof}

Let $X$ be a closed subscheme of a weighted projective space $\PP$. 
Let $\AAA=\CC[x_1,\cdots,x_n]$ be the polynomial ring of $\PP$ and 
$I\subset\AAA$
the saturated homogeneous ideal of $X$.
\begin{theorem}\label{gradedringdeformation} If $depth\AAA/I\geq 2$, 
then the deformation theory of the embedded projective scheme $X$ is 
the same as
that of the corresponding graded ring $\AAA/I$.
\end{theorem}
\begin{proof}
  For the case of usual projective space, we refer to 
\cite[Prop.~8.8]{Ha2}; for the weighted case, one can adapt the same 
proof as [ibid].
\end{proof} Thus in good situations, in particular in ours, we can 
study the  deformations of weighted projective schemes via the 
deformations of
their graded rings, indeed the parameter spaces we shall use will be locally dominating the Hilbert scheme. 

\begin{theorem}[Hilbert scheme]
  There is a natural scheme parametrizing the set of closed subschemes of 
a fixed weighted projective space having a given Hilbert polynomial.
\end{theorem}
\begin{proof}  This result is due to Grothendieck \cite{grothendieck}, and has been
generalized in the multigraded case by \cite{sturmfels}.
\end{proof}

The role of the Hilbert scheme becomes more apparent when we shall take the deformation theory
of the cone  $Cone(H\cap X)$ over the hyperplane section of the canonical model $X$ in its half-canonical 
embedding in a weighted projective space. 

Here we shall use the results of Schlessinger and Pinkham (\cite{schlessinger}, \cite{pinkham}),
in particular Pinkham's result 

\begin{theorem}[Hilbert scheme and deformations]
  The natural morphism of the Hilbert scheme to the  Kuranishi space of $Cone(H\cap X)$ is smooth.
\end{theorem}

It is important to remark that, whereas the Kuranishi  family is versal at any point, it is only semiuniversal
at the point corresponding to  $Cone(H\cap X)$: this is due to the $\CC^*$ action which stabilizes
the cone  $Cone(H\cap X)$ but not its small deformations.

Observe finally  that the local structure of the Gieseker moduli space is the quotient of the
Kuranishi  space of $X$ by the finite group $Aut(X)$ (see \cite{catanese} as a general reference), hence, in order to study the irreducible components of the moduli space,
openness questions are reduced to the study of the Kuranishi  space, in turn this is locally dominated
by the Hilbert scheme (or any parameter space dominating the latter).

\section{Oliverio's surfaces} Oliverio \cite{O} studied the even 
surfaces of general type with $K_S^2=8$, $p_g=4$ and $q=0$ whose 
canonical system is
base point free, showing that their canonical models are the general 
complete intersections  of bidegree $(6,6)$ in the weighted 
projective space
$\PP(1^2,2,3^2)$.

  Let $S$ be an even surface of general type with $K_S^2=8$, $p_g=4$ 
and $q=0$.  Let $L$ be a
half-canonical divisor, that is, $2L=K_S$.  We recall one more 
preliminary result.

\begin{lemma}\label{lispencil}
\begin{enumerate}
  \item[(i)] For any $k\in\mathbb Z$, $h^1(S,kL)=0$.
\item[(ii)] $h^0(S,L)=2$, $h^0(S,2L)=4$, $h^0(S,kL)=k^2-2k+5$ for $k\geq 3$.
\end{enumerate}

\end{lemma}
\begin{proof} This is the content of Lemma 2.1 and Lemma 2.2 of \cite{O}.
\end{proof}

\begin{remark}
It shall be our standard notation that $x_1, x_2$ be a basis of 
$H^0(S,L)$, while $y$ completes $ w_0 := x_1 x_2,w_1 := x_1^2, w_2 
:=x_2^2$ to a basis of
$H^0(S,2L)$,
$z_1, z_2$ complete $$x_1^3, x_1^2 x_2,  x_1 x_2^2,  x_2^3, y x_1, y 
x_2$$ to a basis of $H^0(S,3L)$.

Observe that, since the canonical map of $S$ can't be composed with a pencil (see \cite[Corollary 1 of Theorem 5.1]{Xiao} or \cite[Theorem 3.4]{CanPen} for the case of genus 2 fibrations, \cite{beau} for the higher genus case), this implies that $\phi_K (S)$ is a quadric cone $\{ w| w_1 w_2 = w_0^2 \}$.
\end{remark}

Conversely to the above remark, we show a result of independent interest:

\begin{proposition}\label{cone}
Let $S$ be a minimal surface with $K_S^2=8$, $p_g=4$ and $q=0$ and 
assume that the image of the canonical map $\phi_K (S)$ is a quadric 
cone.
Then,  denoting by $L$ the pull back of a line on the quadric cone, 
we have $ K_S = 2 L  + F$ where
\begin{enumerate}
\item
either $F = 0$ and $S$ is even, or
\item
$ L^2 = 1$ , $K_S \cdot L =3$, $ F \cdot L = 1$,  $|L|$ is a pencil 
of genus $3$ curves having a simple base point, and the canonical map 
has degree $3$, or
\item
$ L^2 = 0$ , $K_S \cdot L = F \cdot L = 2$,  $|L|$ is a base point 
free pencil of curves of genus $g = 2$,  and the degree of $\phi_K$ 
is equal to 2, or
\item
$ L^2 = 0$ ,  $K_S \cdot L  = F
\cdot L = 4$, hence $|L|$ is a base point free pencil of curves of 
genus  $ g = 3$ and the degree of $\phi_K$ is equal to 4.

\end{enumerate}
\end{proposition}
\begin{proof}
As usual write $ K_S = 2 L  + F$, and observe that $K_S$ and $L$ are 
nef divisors. Then
$$ 8 = K_S^2 = 2 L K_S + F K_S = 4 L^2 + 2 L F + FK_ S \geq  4 L^2 + 2 L 
F \geq 4 L^2. $$
The above shows that $ L^2 \in \{0,1,2\}$ and that $FK_ S  $ is even.

If $L^2 = 2$, then $ F K_ S  = 2L F = F^2 = 0$: hence first $F$ is a sum 
of $(-2)$ curves, and then we get $F=0$ since the
intersection form  is strictly 
negative definite on the set of divisors which are sums of $(-2)$ curves.

If  $L^2 = 1$,  then $K_S L $ is odd, at least $3$, but $K_S L = 2 + 
LF $, whence by the above inequality
$ L F = 1$ and $ K_S F =2, F^2 = 0$. Hence $|L|$ is a pencil of curves 
of genus  $3$, and the degree of $\phi_K$ is then
the intersection number  $K_S L = 3$.

If instead $L^2 = 0$, then $ L K_S  = LF$ is even and  $\geq 2$, so 
there are only the possibilities $ L K_S  = LF = 2$,
or $ L K_{S}  = LF = 4.$ In the latter case  $ F K_S = 0$, thus $F$ is a 
sum of $(-2)$ curves, in the former we have
$4 =  F K_S = 4 + F^2$, hence $F^2 = 0$, $F K_S = 4 $. Again  the 
degree of $\phi_K$ equals
the intersection number  $K_S L$.

\end{proof}

Observe that examples of all the above cases have been found (see 
\cite{MP}, \cite{BP}).

We turn now to a  slight improvement of  Oliverio's  result 
concerning the first case of proposition \ref{cone}.

\begin{theorem}\label{fablemma} The canonical models $X$ of even 
surfaces $S$ of general type with $K_S^2=8$, $p_g=4$ and $q=0$ whose 
canonical system is base
point  free are exactly the complete intersections of type $(6,6)$ in the 
weighted projective space $\PP(1^2,2,3^2)$  with at  worst rational 
double points as
singularities. In particular, there are coordinates 
$x_1,x_2,y,z_1,z_2$ such that $X$ is defined by equations of the type
$$ (**) f = z_1 ^2 + z_2 A (x_1,x_2,y) + F (x_1,x_2,y)= 0 , \ \  f' = 
z_2 ^2 +  z_1  A' (x_1,x_2,y) + F' (x_1,x_2,y)= 0. $$
These surfaces form an open set in an irreducible unirational 
component of dimension $35$ of the moduli
space  of  surfaces of general type.
\end{theorem}
\begin{proof}
\cite[Theorem 5.2]{O} proves that  minimal even surfaces of general 
type with $K_S^2=8$, $p_g=4$ and $q=0$  whose canonical system is 
base point free
have a canonical model $X$ which is a complete intersection of type $(6,6)$ 
in the weighted projective space $\PP(1^2,2,3^2)$  (hence $X$ has at 
worse rational double
points as singularities).

Conversely, if $X=\{f=f'=0\}$ is a complete intersection of two 
sextics in $\PP(1^2,2,3^2)$ with only rational double points as 
singularities, then
$K_X=\sO_X(2)$, $K_X^2=8$, $p_g=4$ and $q=0$ and the only nontrivial 
thing to show is that the Weil divisor $L$ such that $ 2 L \equiv K_X$ is indeed a
Cartier divisor (then   the minimal model $S$ is an even surface)  and that the induced linear system is base point free.

To this purpose it suffices to show that 
$|K_X| = |2 L| $ and $ | 3 L|$ are  base point free: because then  $\sO_X(2)$ and $\sO_X(3)$ 
are invertible sheaves, hence also $\sO_X(1)$ is invertible.
Notice the surjection $ H^0 (\hol_{\PP}(2)) \ra  H^0 (\hol_X (K_X))$ 
which implies that the base locus of $|K_X|$ is the locus of zeros of 
all
degree two monomials.

Choose weighted homogeneous coordinates $x_1,x_2,y,z_1,z_2$ in $\PP : = 
\PP(1^2,2,3^2)$ and write
\[
  f=q(z_1,z_2)+g,\ f'=q'(z_1,z_2)+g'
\] with $g,g' \in (x_1,x_2,y)$ and $q, q'$ quadratic forms in $z_1, 
z_2$. Assume by contradiction that the base locus of $|K_X|$ is not 
empty, i.e., that
$X \cap \{ x_1 = x_2 = y =0\} \neq \emptyset$.  Then $q$
and
$q'$ have a common factor, say
$z_1$, and  the canonical system of the surface has a base point at 
$Q=[0,0,0,0,1]$.

Note that $Q$ is a singular point of $\PP(1^2,2,3^2)$, a quotient 
singularity of type $\frac{1}{3} (1,1,2,3) $ with Zariski tangent 
space of dimension $7$,
whose basis is given by the monomials
$x_1^ix_2^{3-i}, x_iy, z_1$. Since $\sO_{\PP(1^2,2,3^2)}(6)$ is 
Cartier, the Zariski tangent space of $X$ at $Q$ has then dimension 
at least $7-2=5$,
contradicting the  assumption on the singularities of $X$
(they have Zariski tangent dimension at most $3$).

Similarly, if $ | 3 L|$ were not  base point free, then $X$ would contain the point 
$Q=[0,0,1,0,0]$, a quotient 
singularity of type $\frac{1}{2} (1,1,1,1) $ with Zariski tangent 
space of dimension $10$: this again contradicts the fact 
that the singularities of $X$  have Zariski tangent dimension at most $3$.

Therefore the quadratic forms $q, q'$ have no common factor, and we 
can change coordinates so that the equations of $X$ take the desired 
form (**), where however $A (0,0,1) \neq 0 $ or $A' (0,0,1) \neq 0 $ :
$$ f = z_1 ^2 + z_2 A (x_1,x_2,y) + F (x_1,x_2,y)= 0 , \ \  f' = z_2 
^2 +  z_1 A' (x_1,x_2,y) + F' (x_1,x_2,y)= 0. $$

These equations naturally exhibit $X$ as a 4-1 cover of the quadric 
cone, and  show that our surfaces
form an open set in an irreducible unirational component of dimension 
$35$ of the moduli space (since we have 6 affine parameters for $A, 
A'$, 15
projective parameters  for
$F, F'$ and we divide out by a group of dimension 7, the group of 
automorphisms of the quadric cone $\PP(1,1,2)$; at any rate, the dimension follows also from \cite[Corollary 5.3]{O}).
\end{proof}

\section{The hyperplane section}\label{sectionhyperplane} Let $S$ be 
an even surface of general type with $K_S^2=8$, $p_g=4$ and $q=0$ and 
let $L$ as above
be  a half-canonical divisor.  We shall view the half-canonical 
curves as hyperplane sections
  of a suitable embedding of the canonical model into a weighted 
projective space, and we shall describe the associated graded ring.

\begin{lemma}\label{basepoint} If $|K_S|$ is not base point free, then
\begin{enumerate}
  \item[(i)] $|K_S|$ has two base points $Q,Q'$ with $Q'$ infinitely 
near to $Q$;
\item[(ii)] $|L|$ has two base points $Q,Q''$ with $Q''$ infinitely 
near to $Q$ and $Q''\neq Q'$;
\item[(iii)] a general curve in $|L|$ is smooth of genus $4$.

\end{enumerate}
\end{lemma}
\begin{proof} These results are in Lemma 3.1, Lemma 3.3 and Theorem 
4.1 of \cite{O}.
\end{proof}

Let us fix a smooth curve $C\in |L|$. Then $Q\in C$ and $\mathcal 
O_C(L)=\mathcal O_C(2Q)$ by  Lemma \ref{basepoint}. Let $x_1$ be a 
section in
$H^0(S,L)$ with $div(x_1)=C$. Using the long  exact   cohomology 
sequence associated to the exact sequence
\[
   0\rightarrow \mathcal O_S((k-1)L)\xrightarrow{x_1} \mathcal 
O_S(kL)\rightarrow \mathcal O_C(2kQ)\rightarrow 0,
\] we have, by Lemma~\ref{lispencil}(i),  the exact sequence
\begin{equation}\label{quotientsequence}
   0\rightarrow H^0(S,(k-1)L)\xrightarrow{x_1} H^0(S,kL)\rightarrow 
H^0(C,2kQ)\rightarrow 0
\end{equation} and it follows that
  \[
   R(S,L)/(x_1)=R(C,2Q)=R(C,Q)^{(2)}\subset R(C,Q),
  \] where $R(C,Q)^{(2)}:=\oplus_k R(C,Q)_{2k}$ is the even part of 
$R(C,Q)$.  In other words, $\Proj R(C,2Q)$ turns out to be a weighted 
hyperplane of
$\Proj R(S,L)$.

\begin{lemma}\label{rcp}
  \[
   R(C,Q)=\CC[\xi,\eta,\zeta]/(p)
  \] where $deg(\xi,\eta,\zeta)=(1,3,5)$ and $p$ is a weighted 
polynomial of degree $15$.  Moreover one can, up to automorphisms of 
$\PP(1,3,5)$,
assume that
\[
  p=\zeta^3-\eta^5+\xi p'
\] where $ p'$ is some suitable weighted polynomial of degree $14$.
\end{lemma}
\begin{proof} We  first calculate $h^0(C,mQ)$ for all $m\geq 0$. By 
the exact sequence (\ref{quotientsequence}), we have
\[
  h^0(C,2kQ) = h^0(S,kL)-h^0(S,(k-1)L)
\]
  for any $k\geq 0$. Together with Lemma \ref{lispencil}, this yields 
in particular
\[
  h^0(C,\sO_C)=h^0(C,2Q)=1, \quad h^0(C,4Q) =2,\quad h^0(C,6Q)=4,
\] which in turns implies  $h^0(C,Q)=1,\ \  h^0(C,5Q)=3$. 

Since $Q$ is a base point of 
$|K_S|$,  it is also a
base point of
$|4Q|= |{K_S}_{|C}| = {|K_S|}_{|C}$, which implies that
\[ h^0(C,3Q) = h^0(C,4Q) =2.
\] For $m\geq 7$, we have $h^0(C,mQ)= m-3 $ by the Riemann--Roch theorem.

Now take a nonzero section $\xi\in H^0(C,Q)$. Then
\[
  H^0(C,Q)=< \xi >,\quad H^0(C,2Q)=< \xi^2 >
\]
  and $\xi$ has only a simple zero at $Q$. There are sections $\eta\in 
H^0(C,3Q)\setminus\xi H^0(C,2Q)$ and
$\zeta\in H^0(C,5Q)\setminus\xi H^0(C,4Q)$.

Since both $\eta$ and $\zeta$ do not vanish at $Q$, there exist 
nonzero $a,b\in\CC$ such that $a\eta^5-b\zeta^3$ vanishes at $Q$. 
Therefore there is
a polynomial $p'$ in $\xi,\eta,\zeta$ of degree 14 such that
\[
  a\eta^5-b\zeta^3 = \xi p'.
\] Up to  rescaling the generators, we have $p = \eta^5- \zeta^3 + \xi p' =0$.

Now, $\xi, \eta, \zeta$ give a morphism into $\PP(1,3,5)$. Therefore 
the image is an irreducible  curve and there are no other relations 
than $p$ holding
among the three elements $\xi, \eta, \zeta$.

In other words we get by pull back  an injective ring homomorphism 
from  $\CC[\xi,\eta,\zeta]/(p)$ to $R(C,Q)$. Since, by the first 
part of our proof,
both rings have the same Hilbert function, they are isomorphic.

\end{proof}

\begin{remark}\label{choicef} Conversely a general curve $C=\{p=0\}$ 
of degree $15$ in $\PP(1,3,5)$ is smooth of genus
$1+\frac{15(15-9)}{2\cdot 15}=4$ and $R(C,Q)$ is naturally isomorphic 
to $\CC[\xi,\eta,\zeta]/p$,  where $\{Q\}:=C\cap(\xi=0)$ is a 
Weierstra\ss\
point whose semigroup is generated by $3$ and $5$;  the proof of 
Lemma \ref{rcp} shows that every smooth curve of genus $4$ with a 
Weierstra\ss\
point  of this form arises in this way.

The smooth curves of degree $15$ in $\PP(1,3,5)$ form a linear system 
of dimension $12$. Since $\dim \Aut \PP(1,3,5)=5$,  they form  a 
subvariety of
dimension $7$ in the moduli space of curves of genus $4$; note that 
it is a divisor in the  locus of the curves whose canonical image is 
contained in
a quadric cone,  which are those possessing only one $g^1_3$.
\end{remark}

\begin{proposition}\label{rc2p}
  \[
   R(C,2Q)=\CC[x_2,y,z_1,z_2,u,v]/I
  \] where $deg(x_2,y,z_1,z_2,u,v)=(1,2,3,3,4,5)$, and $I$ is 
generated by the equations
\begin{align*}\label{relation}
   f_1 & = x_2 z_2 - y^2 & \deg 4\\
   f_2 & = x_2 u- y z_1 & \deg 5\\
   f_3 & = y u - z_1 z_2 & \deg 6\\
   f_4 & = x_2 v - z_1^2 & \deg 6\\
   f_5 & = y v - z_1 u & \deg 7 \\
   f_6 & = z_2 v - u^2 & \deg 8\\
   f_7 & = z_1 A -  y B + x_2 D & \deg 8\\
   f_8 & =  u A -  z_2B + y D & \deg 9\\
   f_9 & = vA -  u B + z_1 D & \deg 10
\end{align*} Here $A,B,D$ are general polynomials of respective 
degrees $5$, $6$ and $7$. Up to automorphisms,  one can assume $A=v$, 
$B=z_2^2$.
\end{proposition}

\begin{proof} We have shown that the graded ring $R(C,Q)$ corresponds 
to a projectively normal embedding in the
weighted projective plane $\PP(1,3,5)$. Therefore the subring 
$R(C,2Q)$ is the even degree part of $R(C,Q)$,  a quotient of the graded ring of the Veronese embedding of the 
plane $\PP(1,3,5)$.

In other words, since the three generators $\xi,\eta,\zeta$ of 
$R(C,Q)$ have odd degrees, the even part $R(C,2Q)$ is generated by 
the six products
\[\begin{matrix}
    x_2:=\xi^2, & y:=\xi\eta, & z_1:=\xi\zeta, \\
    z_2:=\eta^2, & u:=\eta\zeta, & v:=\zeta^2.
   \end{matrix}
\] These in fact define a closed embedding 
$\varphi:\PP(1,3,5)\rightarrow \PP(1,2,3^2,4,5)$. Generators of the 
ideal of $\varphi(\PP(1,3,5))$ are the
$2\times 2$ minors of the $3\times 3$ symmetric matrix
\[
  \left(
  \begin{matrix}
    x_2 & y & z_1\\
    y & z_2 & u \\
    z_1 & u & v
  \end{matrix} \right).
\] Note that $C$ is the curve defined by $p=0$ in $\PP(1,3,5)$.  So 
the ideal  $I$ of $\varphi(C)$ is generated by the defining equations 
of
$\varphi(\PP(1,3,5))$ plus $\xi p, \eta p, \zeta p$.

In view of Lemma~\ref{rcp}, we can write these three homogeneous 
polynomials in terms of
$x_2,y,z_1,z_2,u,v$:
\begin{align*}
  \xi p_{15} &= z_1 v - y z_2^2 + x_2 D \\
  \eta p_{15} &= u v - z_2^3 + y D \\
  \zeta p_{15} & = v^2  - z_2^2 u + z_1 D
\end{align*}

\end{proof} Note that $R(C,2Q)$ is Cohen-Macaulay. In fact for any 
smooth projective curve $C$ and an ample line bundle $H$,  the graded 
ring
$R(C,H)$ is Cohen-Macaulay (see \cite[Prop. 8.6]{Ha2} and its proof).

\section{Two families of surfaces}\label{sectionsyzygy} Recall that 
$R := R(S,L)/(x_1)=  R(C,2Q)$, where $x_1$ is an element of 
$H^0(S,L)$ defining
the curve $C$. The hyperplane section principle \cite[Prop. 1.2]{R1} 
gives the following, which is the explicit counterpart of the 
existence of a flat
1-dimensional family induced by the function $x_1$:
\begin{enumerate}
  \item[(i)] $R(S,L)$ needs exactly one more generator, namely $x_1$, 
and the other generators are lifted from $R$;
\item[(ii)] the relations $F_1, \dots F_9$ among the generators  of 
$R(S,L)$ are liftings of  $f_1,\cdots,f_9$;
\item[(iii)] moreover the first syzygies among the $f_i$ are lifted 
to those among the  $F_i$.
\end{enumerate}

Point (iii) is the tricky part of the principle, and is where 
``formats" are useful in order to write explicitly a flat family 
having as  basis a
locally closed set of an affine space. A format is simply a 
way to write an ideal in such a way that the obvious first syzygies 
are all
the first syzygies (in other words, one produces automatically a flat 
family).

We will describe two formats. Each of them will produce a family of 
minimal surfaces of general type with $p_g=4$, $q=0$, $K^2=8$ and
  even canonical divisor.

\subsection{The extrasymmetric format}\label{pfaffiansection} This 
format was first introduced by M. Reid and D. Dicks (see \cite{R1}, and
\cite{BCP1,BCP2} for further applications and a discussion).

Consider a $6\times 6$ skewsymmetric {\it extrasymmetric} matrix
\[\tilde N=\left(
  \begin{matrix}
    & n_1 & n_2 & n_3 & n_4 & n_5\\
    &  & n_6 & n_7 & n_8 & n_4 \\
   & &  & n_9 & an_7 & an_3 \\
   &  &  &  & bn_6 & bn_2\\
     &  & &  &  & abn_1\\
   &&&&&
\end{matrix} \right),
\] and let $\sI_{\sE} \subset {\tilde 
\AAA}_{\sE}:=\CC[n_1,n_2,\ldots,n_9,a,b]$ be the ideal generated by 
the
$4 \times 4$ pfaffians of $\tilde N$.  A minimal system of generators 
of $\sI_{\sE}$ is given by $9$ of these pfaffians, the last $6$ being 
just
repetitions of simple multiples of them. This $9$ generators are 
yoked by exactly $16$ independent syzygies,  which we can  explicitly
compute (see also
\cite[5.5]{R1}).

\begin{definition} Let $\AAA$ be any weighted polynomial ring and 
consider a ring homomorphism $\varphi \colon {\tilde \AAA}_{\sE} 
\rightarrow
\AAA$; then the ideal $I$ generated by $\varphi(\sI_{\sE})$ is 
generated by the $4 \times 4$ pfaffians  of the $6 \times 6$ 
skewsymmetric matrix
$\varphi({\tilde N})$, obtained by ${\tilde N}$ by substituting to 
each entry its image. We will say that $\varphi(\tilde{N})$is an 
extrasymmetric
format for the quotient  ring $\AAA/I$.
\end{definition}

\begin{example}\label{ex: extrasym} Computing the $4 \times 4$ 
pfaffians of the matrix
\[N=\left(
  \begin{matrix}
    & A & B & z_1 & y & x_2\\
    &  &D & u & z_2 & y \\
   & &  & v & u & z_1 \\
   &  &  &  & 0 & 0\\
     &  & &  &  & 0\\
   &&&&&
\end{matrix} \right),
\] the reader can check that it is an extrasymmetric format 
for $R$. Here
$\AAA= \CC[x_2,y,z_1,z_2,u,v]$ with the grading given in Proposition 
$\ref{rc2p}$.

\end{example}

Let us consider, in  Example \ref{ex: extrasym}, ${\tilde 
\AAA}_{\sE}$ graded by the grading making $\varphi$ a graded 
homomorphism. Since
$\varphi$ is surjective, it yields an isomorphism of graded rings 
$\AAA\cong \tilde \AAA_{\sE}/\ker \varphi$ and
\[
  \ker \varphi= (a-1,b,n_1-{\tilde A}(n_i),n_2-\tilde{ B}(n_i), 
n_6-{\tilde D}(n_i))
\] where $\tilde A$, $\tilde B$, $\tilde D$ are obtained by $A$, $B$, 
$D$ replacing the variables $x_2,y,z_1,z_2,u,v$ by 
$n_5,n_4,n_3,n_8,n_7,n_9$
respectively.

Consider $ \tilde R={\tilde \AAA}_{\sE}/\sI_{\sE}$, and write $\tilde 
f_1,\cdots,\tilde f_9$ for the nine pfaffians of $\tilde N$ generating
$\sI_{\sE}$. Here we can arrange the indices so that $\varphi(\tilde 
f_i)=f_i$ for $1\leq i\leq 9$.  Note that $R = \tilde 
R\otimes_{{\tilde
\AAA}_{\sE}}\AAA$ is such that $ Spec (R)$ a codimension five 
complete intersection in  $Spec (\tilde R)$.
\begin{lemma}\label{tor}
$Tor^{{\tilde \AAA}_{\sE}}_1(\AAA, \tilde R)=0$.
\end{lemma}

\begin{proof} 
This follows since  $\AAA\cong \tilde \AAA_{\sE}/\ker \varphi$  and $  \ker \varphi $ is generated by a regular sequence, 
 see   \cite{matsumura}).

\end{proof}

\begin{corollary} The first syzygy module of $R$ is a reduction of 
the one of $\tilde R$.
\end{corollary}
\begin{proof} Let $ \tilde L_{\bullet}\rightarrow \tilde R\rightarrow 
0$ be a free resolution of $\tilde R$ over
${\tilde \AAA}_{\sE}$. By Lemma \ref{tor}, $\AAA \otimes_{{\tilde 
\AAA}_{\sE}}\tilde  L_{\bullet}\rightarrow \AAA\otimes_{{\tilde 
\AAA}_{\sE}}\tilde
R\rightarrow 0$ is exact at
$\AAA\otimes_{{\tilde \AAA}_{\sE}}\tilde L_1$, which implies the corollary.
\end{proof}

Therefore, to calculate the syzygy module of $R$, it suffices to work 
out that of $\tilde R$.

\begin{corollary}\label{syzygy}The syzygies
\begin{align*} & \sigma_1: -z_1 f_1 + y  f_2 - x_2 f_3 = 0  & \deg 
7\\ & \sigma_2: -u f_1 + z_2  f_2 - y f_3 = 0  & \deg 8\\ & \sigma_3: 
z_1  f_2 - y
f_4 + x_2  f_5 = 0    & \deg 8\\ & \sigma_4: v f_1 + z_1  f_3 -z_2 
f_4 + y  f_5 = 0 & \deg 9\\ & \sigma_5: v  f_1 - u  f_2 + y  f_5 - 
x_2  f_6 = 0 &
\deg 9\\ & \sigma_6: v  f_2 - u  f_4 + z_1  f_5  = 0 & \deg 10\\  & 
\sigma_7: -u  f_3 + z_2  f_5 - y  f_6 = 0 & \deg 10\\ & \sigma_8: -v 
f_3 + u  f_5
- z_1  f_6 = 0 & \deg 11\\ & \sigma_9: B  f_1 - A  f_2 - y  f_7 + x_2 
f_8 = 0 & \deg 10\\ & \sigma_{10}: -B  f_2  + A  f_4 + z_1  f_7 - x_2 
f_9 = 0
& \deg 11\\ & \sigma_{11}:  D  f_1 - A  f_3 - z_2  f_7 + y  f_8 = 0 & 
\deg 11\\ & \sigma_{12}: B  f_3 - A  f_5 - z_1  f_8 + y  f_9 = 0 & 
\deg 12\\ &
\sigma_{13}: -D  f_2 + A  f_5 + u  f_7 - y  f_9 = 0 & \deg  12\\ & 
\sigma_{14}: D  f_3 - A  f_6 - u  f_8 + z_2  f_9 = 0 & \deg 13\\ & 
\sigma_{15}: -D
f_4 + B  f_5 + v  f_7 - z_1  f_9 = 0 & \deg 13\\ & \sigma_{16}: D 
f_5 - B  f_6 - v  f_8 + u  f_9 = 0 & \deg 14
\end{align*} generate all the syzygies between $f_1,\dots,f_9$.
\end{corollary}

\begin{proof}
This follows by Corollary \ref{syzygy} and by the computation of the relations among the pfaffians of $\tilde N$ done  in   \cite[5.5]{R1}.
\end{proof}

This is useful because we can then easily construct flat deformations 
of this ring. Indeed, if we lift  the matrix to a bigger ring $\BB$, 
we will get
automatically a new ideal in $\BB$ generated by lifts of $I$,  and 
also the first syzygies of $R$ will automatically lift, yielding 
flatness. More
formally
\begin{corollary}\label{flexibility} Consider the map $\varphi \colon 
\AAA_{\sE} \rightarrow \AAA$ in Example \ref{ex: extrasym}, a 
surjective  ring
homomorphism $\pi \colon \BB \rightarrow \AAA$, and a ring 
homomorphism $\psi\colon \AAA_{\sE} \rightarrow \BB$  such that 
$\varphi=\pi \circ \psi$
({\it i.e. ``$\psi$ lifts $\varphi$"}).

Let $F_1, \ldots, F_9$ be the nine pfaffians of $\tilde{N}$ such that 
$\pi(F_i)=f_i$.  Then every relation among the $f_i$ lifts to a 
relation among
the $F_i$.
\end{corollary}
\begin{proof} Each relation $\sigma_j$ is obtained by applying $\varphi$ to a 
relation $\tilde \sigma_j$ among the $\tilde f_j$.  Applying $\psi$ to 
the same
relations will give the required lifts.
\end{proof}

By the hyperplane section principle, we can use the above to construct a 
family of surfaces.

\begin{proposition}\label{ME} Consider the extrasymmetric matrix
\[\sN =\left(
  \begin{matrix}
    & \sA & \sB & z_1 & y & x_2\\
    &  &\sD & u & z_2 & y \\
   & &  & v & u & z_1 \\
   &  &  &  & 0 & 0\\
     &  & &  &  & 0\\
   &&&&&
\end{matrix} \right),
\] where $\sA$, $\sB$, $\sD$ are weighted polynomials of respective 
degrees $5$, $6$ and $7$ in the graded polynomial ring
$\CC[x_1,x_2,y,z_1,z_2,u,v]$ with weights $(1,1,2,3,3,4,5)$, and let 
$X \subset \PP(1^2,2,3^2,4,5)$ be given by the vanishing of the $4 
\times 4$
pfaffians of $\sN$.

Then, for general choice of  $\sA$, $\sB$, $\sD$, $X$ is a surface 
with at worse rational double points as  singularities. In this case 
$K^2_X=8$,
$p_g(X)=4$, $q=0$, $K_X=\sO_X(2)$ and $|K_X|$ is not base point free. 
We obtain in this way  a $35$-dimensional unirational family $M_{\sE}$ in the 
moduli
space of  surfaces of general type.
\end{proposition}

\begin{proof} Varying $\sA$, $\sB$ and $\sD$ we let $X$ move in a fixed 
threefold,  the cone over the weighted Veronese surface (its ideal is generated
by  the $2 \times 2$ minors 
of the $3
\times 3$ submatrix on the  top-right corner of $\sN$) which has a 
single point which is not quasismooth, with coordinates $[1,0,0,0,0,0,0]$. The 
same point is
also the only base point of the linear system. By Bertini's theorem, the general surface $X$ 
is singular away from that point. If the coefficient of $x_1^7$ in 
$\sD$ does not
vanish, then this point is a node of $X$.

By Corollary \ref{flexibility} and its proof there is a Gorenstein 
symmetric free resolution of the  ideal of $X$ which lifts a 
resolution of the
ideal of the curve $C=\{x_1=0\}\cap X$ in $\PP(1,2,3^2,4,5)$: both 
are images by a suitable ring map of the Gorenstein symmetric 
resolution of
$\sI_{\sE}$. Since $K_C=\sO_C(3)$, then
$K_X=\sO_S(2)$, and it follows immediately that  the invariants are as stated.

We show that the canonical system of $S$ has a base point. Indeed, in 
the 2-plane $\{x_1=x_2=y=z_1=0\}$ the equations reduce to asking that 
the rank
of the matrix
$$
\begin{pmatrix} z_2&u&\sA\\ u&v&\sB\\
\end{pmatrix}
$$ is not $2$. Such  a determinantal condition defines a locus of 
codimension at most $2$, and with non-trivial cohomology class, hence not empty.

To compute the dimension of the family, we note that, on $X$, 
$z_2=\frac{y^2}{x_2}$, $u=\frac{yz_1}{x_2}$,
$v=\frac{z_1^2}{x_2}$. Then, forgetting the variables $z_2,u$ and $v$, we get a projection map
$\pi \colon X \dashrightarrow \PP(1^2,2,3)$ which is birational onto 
its image, a surface $Y$ of degree
$10$ whose equation  is general in the ideal
\begin{equation}\label{I10} (y^5, x_2y^3, x_2y^2z_1, x_2yz_1^2, 
x_2z_1^3, x_2^2y, x_2^2z_1, x_2^3).
\end{equation} These surfaces $Y$ belong to a family depending on $47$ free parameters, so $Y = 
\pi(X)$ varies in
a 
$46$ dimensional family. We have to subtract from this dimension  the dimension of 
the subgroup
of  $\Aut \PP(1^2,2,3)$ preserving the ideal (\ref{I10}).

Note that $Y = \pi(X)$ has a point of multiplicity $ \geq 3$ at the point $p$ of coordinates
$[x_1,x_2,y,z_1]=[1,0,0,0]$.

The subgroup of automorphisms of $\Aut \PP(1^2,2,3)$ 
preserving the ideal (\ref{I10})  is  exactly, as one can verify, the isotropy group of $p$, a 
group of
dimension $11$. Finally we obtain $46-11=35$.
\end{proof}

\subsection{The $MV$ format}

Consider a $5 \times 5$ skewsymmetric matrix $\tilde M$ and a vector 
$\tilde V$ as follows

\[\tilde M=\left(
  \begin{matrix}
    & m_{12} & m_{13} & m_{14} & m_{15} \\
    &  & m_{23} & m_{24} & m_{25} \\
   & &  & m_{34} & m_{35} \\
   &  &  &  & m_{45}\\
   &&&&&
\end{matrix} \right),\ \ \ \tilde V=\left(
  \begin{matrix} v_1 \\ v_2 \\ v_3 \\ v_4\\ v_5
\end{matrix} \right),
\] and let $\sI_{\sV} \subset {\tilde 
\AAA}_{\sV}:=\CC[m_{12},\ldots,m_{45},v_1,\ldots,v_5]$ be the ideal 
generated by the $4 \times 4$ pfaffians of
$\tilde M$, and by the entries of $\tilde M \tilde V$.

This gives $10$ polynomials, which form a minimal system of 
generators of $\sI_{\sV}$:
\begin{align*}
  g_1:&-m_{23}m_{45} + m_{24}m_{35} - m_{34}m_{25} \\
  g_2:& m_{13}m_{45} - m_{14}m_{35} + m_{34}m_{15}\\
  g_3:& -m_{12}m_{45} + m_{14}m_{25} - m_{24}m_{15}\\
  g_4:& m_{12}m_{35} - m_{13}m_{25} + m_{23}m_{15}\\
  g_5:& -m_{12}m_{34} + m_{13}m_{24} - m_{23}m_{14}\\
  g_6:& m_{12}v_2 + m_{13}v_3 + m_{14} v_4 + m_{15}v_5\\
  g_7:& m_{12}v_1 - m_{23}v_3 - m_{24} v_4 - m_{25}v_5\\
  g_8:& m_{13}v_1 + m_{23}v_2 - m_{34} v_4 - m_{35}v_5\\
  g_9:& m_{14}v_1 + m_{24}v_2 + m_{34} v_3 - m_{45}v_5\\
 g_{10}:& 
m_{15}v_1 + m_{25}v_2 + m_{35} v_3 + m_{45}v_4
\end{align*}
yoked (c.f. \cite[pages 20 and 21]{CR} ) by $16$ independent first syzygies (i.e., relations):
\begin{eqnarray*}\label{syzMV} 
&m_{12}g_2+  m_{13}g_3+  m_{14}g_4+  m_{15}g_5=0\\ &-m_{12}g_1+ 
m_{23}g_3+  m_{24}g_4+  m_{25}g_5=0\\ &-m_{13}g_1
-m_{23}g_2+ 
m_{34}g_4+  m_{35}g_5=0\\ &-m_{14}g_1 -m_{24}g_2 -m_{34}g_3+ 
m_{45}g_5=0\\ &-m_{15}g_1 -m_{25}g_2 -m_{35}g_3 -m_{45}g_4=0\\ 
&v_5g_4   
-v_4g_5 -m_{23}g_6+  m_{13}g_7 -m_{12}g_8=0\\ &-v_5g_3+ 
v_3g_5 -m_{24}g_6+  m_{14}g_7-m_{12}g_9=0\\ &-v_5g_2+v_2g_5+ 
m_{34}g_6-m_{14}g_8+ 
m_{13}g_9=0\\ & 
-v_5g_1+v_1g_5-m_{34}g_7+m_{24}g_8-m_{23}g_9=0\\ 
&v_4g_3-v_3g_4-m_{25}g_6+m_{15}g_7-m_{12} g_{10}=0\\ &-v_3g_2+ 
v_2g_3+m_{45}g_6-m_{15}g_9+  m_{14} g_{10}=0\\ & 
v_4g_1-v_1g_4-m_{35}g_7+m_{25}g_8-m_{23} g_{10}=0\\ &   -v_3g_1+ 
v_1g_3-m_{45}g_7+m_{25}g_9-m_{24}
g_{10}=0\\ &    v_2g_1+ 
-v_1g_2-m_{45}g_8+  m_{35}g_9-m_{34} g_{10}=0\\ 
&v_4g_2-v_2g_4+m_{35}g_6-m_{15}g_8+m_{13} g_{10}=0\\ &v_1g_6+ 
v_2g_7+     v_3g_8+     v_4g_9+     v_5 g_{10}=0\\
\end{eqnarray*}

\begin{remark}\label{rm: cod5} In the previous cases we had a 
codimension $4$ Gorenstein subscheme of a weighted projective space 
defined by an ideal
with $9$ generators: the ideals $I \subset\AAA$ 
and $\sI_{\sE} \subset \AAA_{\sE}$.

Here we need $10$ generators. 
Moreover, the locus has  codimension $5$: indeed, the $5$ 
pfaffians  describe a codimension $3$ Gorenstein
subscheme, and
 at 
the general point of it $\tilde M$ has rank $2$, so the latter $5$ 
polynomials give locally just two conditions. 

The important point 
for us is that the  number of first syzygies  is $16$, as in the previous cases.

\end{remark}

\begin{definition} Let $\AAA$ be any weighted 
polynomial ring, consider a ring homomorphism 
$\varphi \colon 
{\tilde \AAA}_{\sV} \rightarrow \AAA$, and set $M:=\varphi(\tilde 
M)$, $V=\varphi(\tilde V)$.  Let $I$ be the ideal generated 
by
$\varphi(\sI_{\sV})$; $I$ is generated by the $4 \times 4$ 
pfaffians of $M$ and by $MV=0$. In this situation we will say that 
$(M,V)$ is an $MV$
format for the quotient ring 
$\AAA/I$.
\end{definition}

\begin{example}\label{ex: MV} We 
write an MV format for our ring $R$. 

Again we choose the graded ring 
$\AAA=\CC[x_2,y,z_1,z_2,u,v]$.  By  Proposition
\ref{rc2p} we can 
assume $A=v$. Then the pair  of matrices $(M,V)$, where $M$ is antisymmetric and $V$
is a vector, and with
\[M=\left(
 \begin{matrix}
& v & u & z_2 & D \\
   &  & z_1 & y & B \\
  & &  & 0 & v \\
  &  & 
&  & u\\
  &&&&&
\end{matrix} \right),\ \ \ V=\left(
 \begin{matrix} 
x_2 \\ -y \\ z_1\\ 0\\ 0
\end{matrix} \right),
\] is an $MV$ format 
for $R$.

Indeed if we compute the image of the $10$ generators of 
$\sI_{\sV}$ we get exactly (up to a sign) the polynomials 
$f_i$: 
$\varphi(g_1)=f_5$, $\varphi(g_2)=-f_6$, $\varphi(g_3)=-f_8$, 
$\varphi(g_4)=f_9$, 
$\varphi(g_5)=f_3$, $\varphi(g_6)=-f_5$, 
$\varphi(g_7)=-f_4$, $\varphi(g_8)=-f_2$, $\varphi(g_9)=-f_1$, 
$\varphi(g_{10})=-f_7$. The polynomial which
we obtain twice is 
$f_5$, which is produced twice by $M_1$,  the first row of $M$: 
$-f_5$ equals both $M_1V$ ($=\varphi(g_6)$) and the pfaffian of 
$M$
which ignores it ($-\varphi(g_1)$). 

\end{example}

\begin{lemma} The map $\varphi\colon {\tilde 
\AAA}_{\sV} \rightarrow \AAA$ given by  Example \ref{ex: MV} maps 
the $16$ relations among the $10$
generators  of $\sI_{\sV}$ onto a 
generating system of the relations between 
$f_1,\ldots,f_9$.
\end{lemma}

\begin{proof} This is a straightforward 
computation, comparing the images of these relations with the relations 
in Corollary  \ref{syzygy}.
\end{proof}

In this case a general lift 
of $\varphi$ will not produce a flat family, because the  ideal of 
the general fibre will need $10$ generators. Still, a
useful weaker 
statement holds.

\begin{corollary}\label{semiflexibility} Consider 
the map $\varphi \colon \tilde  \AAA_{\sV} \rightarrow \AAA$ in 
Example \ref{ex: MV}, a surjective 
ring homomorphism $\pi \colon \BB 
\rightarrow \AAA$, and a ring homomorphism $\psi\colon \AAA_{\sV} 
\rightarrow \BB$  such that $\varphi=\pi \circ
\psi$ ({\it i.e. 
``$\psi$ lifts $\varphi$"}). Assume moreover 
$\psi(g_6)=-\psi(g_1)$, 
{\it i.e.} that $\psi(\tilde{M}_1\tilde{V})$ equals the image by 
$\psi$ of the  pfaffian obtained by deleting the first row 
and column of
$\tilde{M}$.

Then $\{\psi(g_i)\}$ has cardinality $9$: denote its 
elements by 
$\pm F_1,\ldots,\pm F_9$ so that $\pi(F_i)=f_i$ (so, 
{\it e.g}, $F_6=-\varphi(g_2)$).

Then every relation among the 
$f_i$ lifts  to a relation among the 
$F_i$.
\end{corollary}
\begin{proof}
$\phi$ maps the $16$ generating 
relations among the $g_i$ to a set of generating relations among the 
$f_i$; $\psi$ maps the same relations to relations among their lifts, the   $F_i$. Since a 
generating system of relations lift, every relation
does.
\end{proof}

This format produces naturally a 
family of surfaces.
\begin{proposition}\label{MV} Consider, in the 
graded polynomial ring $\CC[x_1,x_2,y,z_1,z_2,u,v]$ with weights 
$(1,1,2,3,3,4,5)$,   a number $c_0 \in \CC$, three general homogeneous
polynomials $\sD$, 
$\sB$, and $l$ of respective degrees $7$, $6$ and $1$ of the 
form
\[l=c_1x_1+c_2 
x_2
\]
\[\sB=v\sB_v+u\sB_u+z_2\sB_{z_2}+z_1\sB_{z_1}+y\sB_y+\sB_x,
\]
\[\sD=v\sD_v+u\sD_u+z_2\sD_{z_2}+z_1\sD_{z_1}+y\sD_y+\sD_x.
\]

Then 
consider the pair  of matrices $(\sM,\sV)$ (cf.  Example \ref{ex: MV}) with
\[\sM=\left(
 \begin{matrix}
   & v & 
u & z_2 & \sD \\
   &  & z_1 & y & \sB \\
  & &  & l & 
v+l\sB_y-c_0\sD_y \\
  &  &  &  & u-l\sB_{z_1}+c_0\sD_{z_1}\\
 
&&&&&
\end{matrix} \right),\ \ \ \sV=\left(
 \begin{matrix} x_2 \\ 
-y+l\sB_v-c_0\sD_v \\ z_1+l\sB_u-c_0\sD_u\\ l\sB_{z_2}-c_0\sD_{z_2}\\ 
c_0
\end{matrix} \right),
\]

Assume moreover that  $c_0\sD_x=l\sB_x$.

Let $X \subset \PP(1^2,2,3^2,4,5)$ be the zero locus of the ideal 
generated by the $4 \times 4$ pfaffians of 
$\sM$ and by the entries of 
$\sM\sV$.  We get in this way a reducible family of surfaces with 
reducible base $\sT$. The open subset $\{c_0 \neq 0\} \subset \sT$ is irreducible, as well as its closure $\sT_1$. 

Then, for  a general choice of $\sD$, $\sB$,  $l$ and $c_0$ in $\sT_1$, $X$ is a 
surface with at  worse Du Val
singularities (rational double points). If $X$ has Du Val singularities, then $X$ is the canonical model
of a surface of general type and, if $S$ is the
minimal model of
$X$, then 
 $K^2_S=8$, $p_g(S)=4$, $q=0$, $K_S=\sO_S(2)$, and $S$ is an
even surface. 

 The case $c_0 \neq 0$ gives exactly all the surfaces 
with  base
point free canonical system, described in Theorem 
\ref{fablemma}. $\sT_1$ gives a 
$35$-dimensional irreducible locally closed set $M_{\sF}$, in the
moduli space of 
 surfaces of  general type, which contains the set of Oliverio surfaces. 
Moreover $M_{\sF} \cap M_{\sE}$ is irreducible  of 
dimension
$34$.
\end{proposition}

\begin{remark}
First of all, we may write the polynomials $\sB_v,\dots$ uniquely if we require that
$\sB_x$ is a polynomial only in the variables $x_1,x_2$, 
$\sB_y$ is a polynomial only in the variables $x_1,x_2, y$, and so on, following
the increasing weight order $x_1,x_2,y,z_1,z_2,u,v$.

The parameter space $\sT$ is 
reducible, since in fact the equation 
$$c_0\sD_x=l\sB_x \Leftrightarrow c_0\sD_x=(c_1 x_1 + c_2 x_2) \sB_x$$
is equivalent to requiring either that $c_0 \neq 0$ and then $\sD_x = c_0^{-1}(c_1 x_1 + c_2 x_2)\sB_x $,
or that $c_0 =c_1 = c_2 = 0$, or $ c_0 = \sB_x = 0$. 

The closure of the irreducible affine set $\sT \cap \{c_0 \neq 0\}$ shall be denoted by $\sT_1$,
while $\sT_2 : = \{c_0 =c_1 = c_2 = 0\}$, $\sT_3 : = \{c_0 = \sB_x= 0\}$.

On $\sT_2$, $\sB_x $ and $\sD_x$ are arbitrary, hence $\sT_2 \cap \sT_1$ is the subset
where $\sB_x$ divides $ \sD_x$. Similarly, on 
$\sT_3$, $l  = (c_1 x_1 + c_2 x_2), \sD_x$ are arbitrary, hence $\sT_3 \cap \sT_1$ is the subset
where $l  = (c_1 x_1 + c_2 x_2)$ divides $ \sD_x$. The intersection of the three components $\sT_1 \cap \sT_2 \cap \sT_3$
is easily seen to be equal to $\{c_0 = c_1 = c_2 = \sB_x= 0\}$.

In the above theorem we consider only the first irreducible component $\sT_1$, the 
closure of $\{c_0\neq 0\}$. By the forthcoming 
 Theorem
\ref{thatsallfolks} all the other  surfaces  shall belong to
$M_{\sF} \cup M_{\sE}$.
\end{remark}

\begin{proof}  One can verify that the assumption 
$c_0\sD_x=l\sB_x$  boils down to the fact  that the two equations produced by the first 
row coincide (i.e., the pfaffian of the minor of $\sM$ where one erases the first row and column
equals the first entry of $\sM \sV$). 

Then, by  Corollary
\ref{semiflexibility} we have a flat 
family with base $\sT_1$, giving an irreducible locally closed set of the moduli space, which we 
denote by $M_\sF$. 

We consider the subset $M_{\sE\sF}$ of the 
moduli space of  surfaces of general type given by  the image of
$\{c_0=l=0\}\cap \sT_1$.

We have that  $M_{\sE\sF} \subset 
M_{\sE}$: to write a surface in
$M_{\sE\sF}$ in the format of 
Proposition \ref{ME}  it suffices to take $\sA=v$.

We deduce then 
the existence of a surface $X$ with at most rational double points as 
singularities in 
$M_{\sE\sF}$ as in the proof of Proposition 
\ref{ME}. By flatness and Proposition \ref{ME}, its minimal resolution has $K_S^2=8$, 
$p_g(S)=4$, $q(S)=0$ and $K_S$ is the pull back of $K_X=\sO_X(2)$. 

We compute the dimension of 
$M_{\sE\sF}$, forgetting the variables $z_2,u$ and $v$ and  taking the associated projection
as in the proof of 
Proposition \ref{ME}, and then using 
$z_2=\frac{y^2}{x_2}$,
$u=\frac{yz_1}{x_2}$, 
$v=\frac{z_1^2}{x_2}$. 
The image of $S$ is a surface $\Sigma$ of degree
$10$ general in the 
ideal
\begin{equation*}
\label{I10'} (y^5, x_2y^3, x_2y^2z_1, 
x_2yz_1^2, x_2z_1^3,x_2^2y^2,x_2yz_1,x_2z_1^2, x_2^2x_1^6y, 
x_2^3).
\end{equation*} 

Comparing with the ideal (\ref{I10}), we 
have only $46$ parameters: the monomial we are missing 
is
$x_2^2x_1^5z_1$. Arguing as in the proof of Proposition \ref{ME}, 
$\dim M_{\sE\sF}= 45-11=34=\dim M_{\sE}-1$.

When $c_0\neq 0$, the 
two equations of smaller degree eliminate $u$ and $v$,  embedding the 
surface as a complete intersection of type $(6,6)$ in
$\PP(1^2,2,3^2)$, so, by 
theorem \ref{fablemma},  the canonical system is base point 
free.

Conversely, all the isomorphism classes of canonical models of such complete intersection surfaces 
 are here. Indeed,   choosing for simplicity $c_0=c_1=1$, 
$c_2=0$, $\sD_x=x_1\sB_x$
(to ensure that we are in $\sT_1$) we get 
$u=-D_{z_1}+\ldots$, $v=D_y+\ldots$ and $X$ 
becomes the complete intersection in $\PP(1^2,2,3^3)$ of the two 
sextics which are obtained   eliminating $u,v$ in the sextics 
$yu-z_1z_2 -x_1v$ and $\sB+x_1(y\sB_{ z_2}-z_1\sB_{ u})+\ldots$, where 
the terms  in ``$\cdots$" do not depend on 
$\sD_y$, $\sD_{z_1}$ or 
$\sB$. It follows that we get all pencils of sextics (with base locus 
a surface with at most rational double points of
singularities) 
containing a sextic  in the ideal generated by $x_1$, $y$ and 
$z_1z_2$.  On the other hand, if 
we cannot find such a sextic in the pencil even after a projective coordinate change, then, argueing as in the proof  of Theorem \ref{fablemma}, the 
canonical system has a base point, a contradiction.

Therefore we have shown that the surfaces $X_t$ in the family $ M_\sF$
admit a smooth deformation $X_0$;  $X_0$ is an even surface, because
it is  smooth and  $\hol_X (K_{X_0}) \cong \hol_{X_0} (2)$. Hence all the minimal models
$S_t$ of our canonical models $X_t$, being diffeomorphic to $X_0$, are even surfaces.
 
By Theorem 
\ref{fablemma} the surfaces with  base point free canonical system
form a 
$35$-dimensional irreducible open set of the moduli space, 
so $\overline{M_{\sF}}$ is an irreducible  component of the moduli space,
containing the set of Oliverio surfaces.

Finally, we have already proved that 
$M_\sE \cap M_\sF$ contains the irreducible family 
$M_{\sE\sF}$ of 
dimension $34$. On the other hand, let $X$ be a surface in $M_\sF$ 
which is also in $M_\sE$. Then $K_X$ has base points, so 
$c_0=0$.
Moreover, the equation of degree $4$ must be of the form 
$WZ-Y^2$ for some forms $W,Y,Z$ of respective degree
$1,2$ and $3$: 
this forces $c_1=0$, so $l=c_2x_2$. A long but straightforward 
computation shows that we can make a coordinate change so that 
the
generators are still produced by matrices $\sM$ and $\sV$ as in the 
statement, but with $l=0$. So $X \in 
M_{\sE\sF}$.
\end{proof}

\section{Deformations of the cone and the 
moduli space}\label{sectionconedeform} 

Our next goal is to prove that 
$M_\sE$ and $M_\sF$ fill
$M_{8,4,0}^{ev}$. By  Theorem~\ref{fablemma}  and  Proposition~\ref{MV} it suffices to restrict our considerations to
surfaces $S$  in 
$M_{8,4,0}^{ev}$ such that $|K_S|=|2L|$ is not base point free. 
It will be convenient to consider only the canonical models $X$ of such surfaces,
observing that $K_X=2L$ and $R(X,L)$ is an \emph{extension 
ring} of  degree one (\cite{R1}) of the ring $R=R(C,2Q)$,
where $C$ is a general curve in the pencil $|L|$
(this simply means that $R \cong R(X,L)/ (x_1)$, where the coordinate $x_1$ has degree 1, which is   the 
algebraic counterpart of the geometric  process of taking a hyperplane section). 

First of all, if 
no nonsingularity condition is set forth,  a trivial extension ring of $R \cong \AAA /(f_1,\cdots,f_9)$, where 
$\AAA$ is the polynomial ring $\CC[x_2,y,z_1,z_2,u,v]$,
is given by the cone 
\[
 C_R 
:= Proj (\BB/(f_1,\cdots,f_9))
\]  where 
$\BB$ is the polynomial ring $\CC[x_1,x_2,y,z_1,z_2,u,v]$ 
and
$deg(x_1,x_2,y,z_1,z_2,u,v)=(1,1,2,3,3,4,5)$. 

Every extension 
ring of $R$ can be viewed as a deformation of $C_R$, since in both situations the issue is
to lift  the
same generators $f_1,\cdots,f_9$ of the graded ideal and their  first 
syzygies $\sigma_1,\cdots,\sigma_{16}$. 

Pay attention  
that the ideal $J$ of $C_R$ is different
from the one of $R$, since $J$ 
is generated by $f_1,\cdots, f_9$ in the bigger polynomial ring 
$\BB$. 

 An explicit calculation of the infinitesimal deformations 
of  $C_R$ occupies the main part of this section.

\subsection{First order deformations of 
$C_R$}\label{firstdeformationsection}
As usual, we begin by  calculating 
the space $T^1$  of  first order deformations: since we know that the Kuranishi family is parametrized by a 
complex analytic subspace
of $T^1$.

A first order  deformation 
of $C_R$ (see for instance \cite{schlessinger},  Section 1) is an element of $ Hom (J, \BB / J)$ 
and can be therefore written in  the following form: 

\begin{equation}\label{1storderdeform}
 F_i^{(1)} = f_i + 
\epsilon\cdot\sum_{k\geq0} x_1^k f_{i,k}', \ \ 1\leq i \leq 
9,
\end{equation} where $f_{i,k}'\in\AAA=\CC[x_2,y,z_1,z_2,u,v]$ 
which is viewed as a subring of the polynomial ring $\BB$; so the 
$F_i^{(1)}$'s are
elements in $\BB[\epsilon]$.

A  standard observation is that   each $f_{i,k}'$ can  be viewed as an 
element of $R$ which is a quotient of 
$\AAA$. In fact, supposing that  $F_i^{(1)}$ 
and $G_i^{(1)}$ define two first order deformations of $C_R$ and that
$F_i^{(1)}-G_i^{(1)}$ is in $\epsilon\cdot J$ for $1\leq i\leq
9$: 
then they actually define the same first order 
deformation, since $F_i^{(1)}$'s and the 
$G_i^{(1)}$'s generate the same ideal of 
$\BB[\epsilon]$.

 For $1\leq j\leq 16$, suppose the relations
(first syzygies) $\sigma_j$ are
 $\sum_{1\leq i\leq9} l_{ij} f_i=0$ (see 
Corollary \ref{syzygy}). Then the relations between the $F_i^{(1)}$ should 
be of the form 
\[
\sum_{1\leq i\leq 9} (l_{ij}+ \epsilon\cdot 
m_{ij})F_i^{(1)}=0 
\] where $m_{ij}$ is an element of $\BB$. The possibility of 
 lifting the  relations is equivalent to the condition that we get a homomorphism
 of $J$ into  $\BB / J$ and yields the exact restrictions on the $f_{i,k}'$ 
in $(\ref{1storderdeform})$.

\begin{lemma} For the first order 
deformations, it suffices to lift the following  five of the sixteen relations (first order syzygies): 
$\sigma_1,\sigma_3,\sigma_5,\sigma_9,\sigma_{10}$.
\end{lemma}
\begin{proof}
 
Indeed, we have the following equivalences (mod $ J$)  between the 
first order syzygies
\begin{align*}
 z_1\sigma_2 &\equiv u \sigma_1 
& y\sigma_{11} &\equiv v\sigma_2 + z_2\sigma_9\\
 y \sigma_4 &\equiv 
-z_1 \sigma_2 + z_2 \sigma_3       & x_2\sigma_{12} &\equiv 
-(z_2^2\sigma_1 + v\sigma_3 + z_1\sigma_9 + y\sigma_{10})\\
 z_1 
\sigma_6 &\equiv v \sigma_3                       & z_1\sigma_{13} 
&\equiv v\sigma_6 + u\sigma_{10}\\
 x_2 \sigma_8 &\equiv v \sigma_1 + 
z_1 \sigma_5         & z_1\sigma_{14} &\equiv v \sigma_8 + 
u\sigma_{12} \\
 v\sigma_7 &\equiv u \sigma_8 
& z_1\sigma_{15} &\equiv z_2^2 \sigma_6 + v\sigma_{10} \\
  & 
& z_1\sigma_{16} &\equiv z_2^2\sigma_8 + v\sigma_{12}\\
\end{align*} 
Since the variables $x_2,\dots, v$ are not zero-divisors in $C_R$, the 
syzygies $\sigma_1, \sigma_3, \sigma_5$ imply $\sigma_2, 
\sigma_4,
\sigma_6, \sigma_7, \sigma_8$ mod $ J$ by the first column, 
and $\sigma_1, \dots, \sigma_{10}$ imply the remaining ones 
$\sigma_{11},\dots,\sigma_{16}$ by the
second column. 

\end{proof}

We can calculate the $f_{i,k}'$ separately, since they 
correspond to different degrees in $x_1$. Denote by $V_k$ the space of first 
order deformations having   degree
$k$ in $x_1$, and by $V_k'$ the subspace 
of $V_k$ induced by variations of  the entries of the matrix $N$ (while 
preserving the extrasymmetric format) in
Example \ref{ex: extrasym}. 
Due to the previous observations, we  see that 
 the calculations of first order 
deformations essentially take place  in the quotient ring $\BB /J$ 
.
\begin{proposition}\label{firstdeformation}
 \begin{enumerate} 
 
\item [(i)] $V_k = V_k'$ for $k\geq 2$, that is, every 1st order 
deformation of $C_R$ with degree $2$ in $x_1$ is obtained from the 
extrasymmetric
format.
  \item [(ii)] $ dim\ V_1 / V_1' = dim\ V_0 / 
V_0' =1$, that is, in degrees $1$ (resp. $0$) in $x_1$, the  first
order deformations that are 
induced by the matrix format build a subspace of codimension 1.
 
\end{enumerate}
\end{proposition}
\begin{proof} For every $k\geq 0$, 
let $F_i^{(1)} = f_i + \epsilon x_1^k f_{i,k}', \ \ 1\leq i \leq 9$ 
be a first order deformation of $C_R$, with
degree $k$ in $x_1$. We 
will compare $V_k'$ and $V_k$ for each $k\geq 0$.

Note that 
$deg(f_1,\dots,f_9) = (4,5,6,6,7,8,8,9,10)$ and $deg\ f_{i,k}'= deg\ 
f_i - k$. If $k>10$, then $f_{i,k}'=0$ by degree reason. If 
$k=10$,
then 
\[
 f_{1,10}' = \cdots = f_{8,10}'=0. 
\] By the relation  
$\sigma_{10}$, we have
\[ -z_2^2 f_{2,10}'  + v f_{4,10}' + z_1 
f_{7,10}' - x_2 f_{9,10}' = 0,
\] and it follows that $f_{9,10}'$ is 
also $0$. A similar argument shows that, if $k=8$ or $9$, then 
$f_{i,k}'=0$ for $1\leq i \leq 9$. Therefore $V_k =
0$ for $k\geq 8$ 
and  also $V_k' = 0$ a fortiori. 

Next we show that  $V_k / V_k' = 0$ for every $2\leq k \leq 7$. Since the calculations 
are similar, we treat only  the case when $k=7$ and leave the
rest of the verifications 
 to the reader.

Case $k=7$: for degree reasons, one has 
$f_{1,7}'=\dots=f_{4,7}'=0$. The syzygy $\sigma_3$ implies that 
\[
z_1 f_{2,7}' - y f_{4,7}' + x_2 f_{5,7}' = 0 
\] and it follows that 
$f_{5,7}'=0$. In turn we have $f_{6,7}'=0$ by $\sigma_5$ . 

Then 
$\sigma_9$ and $\sigma_{10}$ yield  
\[
  y f_{7,7}'= x_2 
f_{8,7}',\hspace{1cm} z_1 f_{7,7}' = x_2 f_{9,7}'
\]
 which implies 
that $(f_{7,7}',f_{8,7}',f_{9,7}') = (c x_2,cy,cz_1)$ with $c\in\CC$, 
but this infinitesimal deformation  is induced by varying one entry of the $6\times 
6$
antisymmetric matrix $N$: $D\mapsto D + \epsilon cx_1^7$.

For (ii) 
(resp. (iii)), we will show that, up to a scalar, there is exactly 
one first order deformation with degree $1$ (resp. $0$) in $x_1$ 
that
cannot be induced by varying the entries of the matrix $N$.

Let us treat now the case $k=1$: as before, we have that  the degrees 
$$deg(f_{1,1}',\dots,f_{9,1}') = (3,4,5,5,6,7,7,8,9).$$ Using the infinitesimal matrix entry 
changes of the form $z_2 \mapsto z_2
+ \epsilon x_1\cdot (\cdots)$ and $x_2 
\mapsto x_2 + \epsilon a x_1$, we can assume $f_{1,1}' = c_1 z_1$. A similar 
change of the entry  $u$ allows us to assume that $f_{2,1}' =
c_2 u$. Then the 
relation $\sigma_1$ gives 
\[
 -z_1 f_{1,1}' + y f_{2,1}' - x_2 
f_{3,1}' = 0
\] and we see
that $f_{2,1}' = 0,\ f_{3,1}' = - c_1 v$. 
Now making an appropriate change at $v$  we can assume 
$f_{4,1}' = 
c_{4,1} y z_2 + c_{4,2} v$. We have the following equations by 
$\sigma_3,\sigma_5$:
\[ - y f_{4,1}' + x_2 f_{5,1}' = 0, 
\hspace{0.5cm} v f_{1,1}' + y f_{5,1}' - x_2 f_{6,1}' = 0
\] and it 
is not hard to see that 
\[ f_{4,1}' = - c_1 y z_2, \hspace{0.5cm} 
f_{5,1}' = - c_1 z_2^2, \hspace{0.5cm} f_{6,1}' = -c_1 D.
\] Next, up 
to appropriate first order changes of $v,z_2^2,D$ in the top left 
corner of $N$ by multiples of $x_1$, we can assume $f_{7,1}' = c_7 
z_2
u$. The relations $\sigma_9,\sigma_{10}$ imply that
\[
 z_2^2 
f_{1,1}' - y f_{7,1}' + x_2 f_{8,1}' = 0, \hspace{0.5cm}  v f_{4,1}' 
+ z_1 f_{7,1}' - x_2 f_{9,1}' = 0
\] and we find that
\[
 c_7 = c_1, 
\hspace{0.5cm} f_{8,1}' = f_{9,1}' = 0
\] Summing up, we have the 
following first order deformation with degree $1$ in 
$x_1$:
\begin{align*}
 &f_{1,1}' = c_1 z_1,& &f_{3,1}' = - c_1 v,& 
&f_{4,1}' = - c_1 y z_2,& \\ &f_{5,1}' = - c_1 z_2^2,&  &f_{6,1}' = 
-c_1 D,& &f_{7,1}' = c_1 z_2 u,
&f_{2,1}' = f_{8,1}' = f_{9,1}' = 
0
\end{align*} which is evidently not induced by the entry changes of 
$N$.

Case $k=0$: since the calculation is similar to the  case $k=1$, 
we leave it to the reader. The codimension of the space of
first order
deformations in degree $0$ that  come by entry 
changes of the matrix $N$ is one and a basis of $V_0/V_0'$ is represented 
by
\begin{align*} & f_{1,0}' = -c_0 u,&&f_{4,0}' = c_0z_2(d_{0,2} 
z_1-z_2),&& f_{7,0}' = -c_0( D_y z_1 + d_{0,2} z_2 v),\\ &f_{2,0}' = 
-c_0 v,&& f_{5,0}'
= -c_0(\delta x_2^7 + D_y y + D_{z_1}z_1),&& 
f_{8,0}' = c_0 (\delta x_2^6 z_1+D_{z_1}v),\\ &f_{3,0}' = 0, && 
f_{6,0}' = -c_0 (\delta x_2^6 y + D_y
z_2+D_{z_1}u),&& f_{9,0}' = 
-c_0D_y v
\end{align*} where we have decomposed $D$ as $D=\delta 
x_2^7 + D_y y + D_{z_1}z_1+ d_{0,2}z_2u$.

\end{proof}

 Recall that the subspace of elements of non-positive grading in  the space of first order deformations $ Hom (J, \BB /J)$ 
 yields  the tangent space
to the Hilbert scheme at the point corresponding to the cone $C_R$.

However, to calculate the tangent space to the Kuranishi family, we must consider isomorphism classes
of first order deformations, i.e., we must divide by the subspace generated by the action of the
Lie algebra of vector fields on the weighted projective spae, the tangent space
to the group of projective automorphisms.

We divide therefore by these  infinitesimal coordinate 
changes and, using 
Proposition \ref{firstdeformation}, we can assume that any first order deformation of the cone $C_R$ 
is equivalent to one of the
form
\begin{equation}\label{firstorderreduction}
F_i^{(1)}  =  f_i+ \epsilon (f_{i,0}' +  x_1 f_{i,1}' +  \sum_{k\geq 
0} x_1^k f_{i,k}''),\  1\leq i \leq 9
\end{equation} where $f_{i,0}'$ 
(resp. $f_{i,1}'$) is the deformation with degree $0$ (resp. $1$) in 
$x_1$ in the  proof of Proposition
\ref{firstdeformation} and  the $ f_i 
+ \sum_{k\geq 0} x_1^k f_{i,k}''$ $(1\leq  i\leq 9)$ are the $4\times 
4$ Pfaffians of the following 
antisymmetric
matrix:
\begin{equation}\label{firstdeformationmatrix}
\sN^{(1)}= 
\left(
 \begin{matrix}
   &\sA^{(1)} &\sB^{(1)} &  z_1 &  y &  x_2\\
&  &\sD^{(1)} & u & z_2 &  y \\
  & &  &  v &  u &  z_1 \\
  &  &  & 
& 0 & 0\\
    &  & &  &  & 0\\
  &&&&&
\end{matrix} 
\right).
\end{equation} with
\begin{align*}
 \sA^{(1)} = & v+ 
\epsilon\sA'=v + \epsilon a_5x_1^5,\hspace{0.5cm} \\
 \sB^{(1)} = & 
z_2^2+ \epsilon \sB'=z_2^2 + \epsilon(b_1 x_1 v + b_2 x_1^2 u + 
b_3x_1^3 z_2 + b_6 x_1^6),\\
 \sD^{(1)} = & D+ \epsilon \sD'=\delta 
x_2^7 + D_y y + D_{z_1}z_1+ d_{0,2}z_2u+\epsilon(\delta' x_2^7 + D'_y 
y + D'_{z_1}z_1+ d'_{0,2}z_2u) \\
             &\hspace{2cm} 
+\epsilon(d_{1,1} x_1 y u + d_{1,2}x_1 z_2^2 + d_{2,1} x_1^2 y z_2 + 
d_{2,2} x_1^2 v + d_3 x_1^3 u \\
             &\hspace{5.5cm}+ 
d_{4,1} x_1^4 z_1 + d_{4,2}x_1^4 z_2 + d_5 x_1^5 y + d_7 
x_1^7).
\end{align*} so that $\sA^{(1)} - v = \epsilon\sA', \sB^{(1)} 
-  z_2^2 = \epsilon \sB', \sD^{(1)} - D = \epsilon \sD'$ are first 
order infinitesimals.
Here we have thrown as many terms as possible 
from $\sA^{(1)}$ and $\sB^{(1)}$ to $\sD^{(1)}$ using the equations 
$f_1=\cdots=f_9=0$. 

\begin{remark} If we use neither the coordinate 
changes nor the entry changes of $N$ in the proof and keep track of 
the free parameters, then we
obtain all the 
dimensions:
\begin{enumerate}
 \item[$\bullet$] $dim_\CC V_k=0$, for 
$k\geq 8$.
 \item[$\bullet$] 
$dim_\CC\{V_7,\dots,V_0\}=\{1,2,4,6,11,16,23,30\}$.
\end{enumerate} 
Since $V_k  = Hom_R(J/ J^2, R)_{-k}$ (\cite[Theorem 1.10]{R1}), these 
dimensions   can be  calculated  in Macaulay 2 for an
explicitly 
assigned $D$  as follows.

 \small
\begin{verbatim}
--We count here the dimensions of first order deformations 
--of the graded ring R=C[x_2,y,z_1,z_2,u,v]/(f_1,...,f_9) with D=0.

AA=QQ[x_2,y,z_1,z_2,u,v,Degrees=>{1,2,3,3,4,5}];
f_1=x_2*z_2-y^2;
f_2=x_2*u-y*z_1;
f_3=y*u-z_1*z_2;
f_4=x_2*v-z_1^2;
f_5=y*v-z_1*u;
f_6=z_2*v-u^2;
f_7=z_1*v-y*z_2^2;
f_8=u*v-z_2^3;
f_9=v^2-z_2^2*u;
I=ideal(f_1,f_2,f_3,f_4,f_5,f_6,f_7,f_8,f_9);
--the ideal of R;
I'=I/I^2;R=AA/I;
H=Hom(I',R);h=hilbertSeries (H,Order => 1)
--h is the hilbert series of Hom_R(I/I^2,R) up to degree 0;
--the coefficients of h are exactly the dimensions of V_k, 0<=k<=7.
\end{verbatim}

\end{remark}

\subsection{Lifting to higher 
orders}\label{sectionsecondlift}

We shall now try to  lift the first order 
deformations obtained in Section~\ref{firstdeformationsection} to higher order. 
We shall do this in several steps, first of all we shall calculate the tangent cone
to the base $\sB$ of the Kuranishi family (equivalently, to the Hilbert scheme).

We shall in this way obtain some quadratic equations which a posteriori
will be shown to yield the equations defining $\sB$; but since we do not want to use computer assisted calculations,
we shall proceed in steps, first showing that these equations define the tangent cone,
then that these equations define $\sB$ after a possible coordinate change,
and only in the proof of the main theorem we shall see that these equations define   $\sB$ 
in the initially chosen coordinates.

\begin{proposition}\label{secondlift} If a first order deformation 
of $C_R$ as defined in (\ref{firstorderreduction})  lifts to a 
genuine deformation,
then 
$$c_0a_5 = c_1a_5 = c_0d_7- c_1 b_6  \equiv 0 \  ( mod \  \mathfrak M^3),$$
$\mathfrak M$ being the maximal ideal of the origin in the vector space
of first order deformations.
\end{proposition} 

\begin{proof} Starting with $F_i^{(1)}$ 
$(1\leq i\leq 9)$, we can write a one parameter family of deformations of order $n$ as
\[  F_i^{(n)} =  f_i + t f_i^{(1)} + \cdots + t^n f_i^{(n)},\ 1\leq i 
\leq 9.
\] where $f_i^{(1)} = f_{i,0}' +  x_1 f_{i,1}' +  \sum_{k\geq 
0} x_1^k f_{i,k}''$ is the part defining $F_i^{(1)}$  in 
(\ref{firstorderreduction}) and
$t$ is an infinitesimal parameter of 
order $n$ (i.e. $t^{n+1}=0$). Therefore the $F_i^{(n)}$ are elements in 
$\BB[t]/(t^{n+1})$. For $m\leq n$, there is a
natural surjection of 
rings 
\[
 \BB[t]/(t^{n+1})\rightarrow \BB[t]/(t^{m+1}),
\] and we 
denote by $F_i^{(m)}$ the image of $F_i^{(n)}$ in $\BB[t]/(t^{m+1})$. 
The first syzygies $\sigma_1,\cdots,\sigma_{16}$ between the $f_i$ should 
lift
to those between the $F_i^{(m)}$ for any $m\leq n$, so that the $F_i^{(m)}$ 
define a deformation of $C_R$ of order $m$.

The relation $\sigma_1$ 
between $f_i$ lifts to the second order as 
 \begin{align*}
 -z_1 
F_1^{(2)} + y F_2^{(2)}  - x_2 F_3^{(2)}  = &t(c_1 x_1 F_4^{(2)}-c_0 
F_5^{(2)} ) - t^2 c_0^2 (D-d_{0,2}z_2u)+ t^2c_1^2 x_1^2 yz_2\\ & - 
t^2
c_0c_1  d_{0,2}  x_1 z_1 z_2 + t^2(-z_1 f_1^{(2)} + y f_2^{(2)} + 
x_2 f_3^{(2)})
\end{align*}
and from this we see that $f_2^{(2)}$ does 
not contain $x_1^5$.

The first syzygy $\sigma_3$ lifts as 

\begin{align*}\label{thirdsyzygy}
 z_1 F_2^{(2)} - y F_4^{(2)}  + 
x_2 F_5 ^{(2)} = &t(c_0 (d_{0,2} z_2 F_2^{(2)} - F_7^{(2)}) - c_1 x_1 
z_2 F_1^{(2)}) \\
                                  & 
+t^2(c_1^2x_1^2z_1z_2-c_0^2D_yz_1+c_0 ( z_1\sA'  -  y \sB' + x_2 
\sD'))\\
                                  &+t^2(z_1 f_2^{(2)} - y 
f_4^{(2)}  + x_2 f_5 ^{(2)})
\end{align*}

 Note that $t^2c_0z_1\sA'$ 
contains $t^2 c_0a_5x_1^5 z_1$ and $t^2c_0 x_2 \sD'$ contains $t^2 
c_0d_7x_1^7x_2$ (cf. (\ref{firstdeformationmatrix})).
Since $t^2 
c_0a_5x_1^5 z_1$ could only be cancelled by  $t^2z_1f_2^{(2)}$, but 
unfortunately $f_2^{(2)}$ does 
not contain $x_1^5$,  $c_0a_5$ must be $0$. 
Besides
$t^2 c_0d_7x_1^7x_2$ can only be absorbed into 
$t^2x_2f_5^{(2)}$, so the coefficient of $x_1^7$ in $f_5^{(2)}$ is 
$-c_0d_7$.

The relation $\sigma_5$ gives

\begin{align*}
 v F_1^{(2)} 
- u F_2^{(2)} + y F_5^{(2)} - x_2 F_6^{(2)} = & t(c_1 x_1 F_7^{(2)} + 
c_0 (D_yF_1^{(2)} + D_{z_1}F_2^{(2)})\\
& + t^2 (c_0^2(uD_y+vD_{z_1})+c_0c_1d_{0,2} x_1 z_2 v 
+c_1^2x_1^2z_2u\\
 &-t^2c_1x_1(z_1\sA' -  y\sB' + x_2\sD') \\
 &\hspace{2cm}  + t^2(v f_1^{(2)} - u f_2^{(2)} + y f_5^{(2)} - x_2 
f_6^{(2)}) 
\end{align*}

The term $t^2c_1a_5x_1^6 z_1$ in $t^2 
c_1x_1z_1\sA'$ cannot be absorbed in any of $t^2 f_i^{(2)}$, 
$i=1,2,5,6$, hence $c_1a_5 = 0$. On the other hand
the term $t^2c_1 
b_6 x_1^7y$ in $t^2c_1x_1y\sB'$ can only be absorbed into 
$t^2yf_5^{(2)}$, so the coefficient of $x_1^7$ in $f_5^{(2)}$ is 
$-c_1
b_6$. Comparing with the coefficient of $x_1^7$ in $f_5^{(2)}$ 
determined by  $\sigma_3$ above, we obtain
$ c_0d_7 = c_1 b_6 = 0$. 
Summing up, we have the following restrictions on the 
coefficients:
\[
 c_0a_5 = 0,\quad c_1a_5 = 0,\quad c_0d_7=c_1 
b_6.
\]
\end{proof}

For the reader's benefit, we observe that the above equations describe the algebraic set 
$$ \{ a_5 = c_0d_7- c_1  b_6 = 0 \} \cup \{ c_0 = c_1 = 0 \} ,$$
which is not a complete intersection since it has codimension two, while the space of quadrics containing
it has dimension three.

\subsection{The moduli space}
Let's come back 
to our original problem about the moduli space of even surfaces with 
$K^2 = 8$, $p_g = 4$, $q = 0$. 

In the next lemma we are essentially 
continuing with the previous calculations, except that we
set for 
convenience 
$$ a : = a_5, \  d: = d_7, \ b :=  b_6.$$

Our purpose 
is to show that our previous equations, which were among the 
equations defining the tangent cone to the base
of the Kuranishi 
family, are indeed the equations of the base of the Kuranishi 
family.

\begin{lemma}\label{tangentcone}
Set
$$ f_0 : =  c_0 a 
,\quad f_1 : =  c_1 a ,\quad g : = c_0 d  - c_1 b. $$

Let $\sP$ be the polynomial ring $\CC [ a,b,c_0 , c_1,d, u_1, \dots 
u_m]$ and consider
an ideal  $ J \subset \sP$ such that $J$ contains polynomials $F_0, 
F_1, G$ such that
$$ F_0 \equiv  f_0  \  ( mod \ \mathfrak M^3) \ \  F_1 \equiv  f_1  \ 
( mod \ \mathfrak M^3) \ \ G \equiv  g  \  ( mod \ \mathfrak M^3,)$$
$\mathfrak M$ being the maximal ideal of the origin.

Let $W$ the subscheme associated to $J$, and assume that $W$ contains 
two distinct irreducible components of
codimension $2$. Then, up to an analytic change of coordinates, we 
may assume that
$$ (**) J = ( f_0, f_1, g ) , \    W = W_1 \cup W_2 , \ W_1 =  V ( 
c_0, c_1) , \ W_2 = V (a, g).$$

In particular, $W$ is schematically the union of two irreducible 
components of codimension $2$ which are complete intersections.
\end{lemma}
\begin{proof}

We divide the proof in several steps.

{\bf Step 1.} Observe that $W$ is a subscheme of $ V (I)$, where $I$ 
is the ideal $ I : =  ( F_0, F_1, G ) $.

W.L.O.G. we may replace $J$ by $I$. In fact, once the final assertion 
$(**)$ is proven for $I$,
$W$ is a subscheme of $V(I) = W_1 \cup W_2$ containing two 
irreducible components of codimension $2$,
hence $ W =  V(I)$, and the proof is finished.

{\bf Step 2.} Assume henceforth $J = I$. We observe first that $W$ 
cannot have any component of codimension $1$.
Else there would be a function $F$ dividing $F_0, F_1, G $: in 
particular the leading form $f$ of $F$ would divide
$f_0, f_1, g$, therefore $f$ would be a constant , and $F$ would be a 
unit in the local ring of the origin.

Let $Z_i : = V ( F_i , G)$: the same argument shows that neither 
$Z_0$ nor $Z_1$ has a component of codimension $1$,
hence $Z_0, Z_1$ are complete intersections of codimension $2$.

{\bf Step 3.} Consider now the irreducible decomposition $ W = \cup_i 
W_i$ of $W$ into irreducible components.
By what we have seen, each $W_i$ has codimension either $2$ or $3$. 
We assume that $W_1$ and $W_2$
have codimension equal to $2$.

{\bf Step 4.} Consider now the tangent cone $\sC$ of $W$ at the 
origin.  The irreducible decomposition $ W = \cup_i W_i$ of $W$
yields a decomposition $ \sC = \cup_i \sC_i$.
Observe moreover that $\sC$ is a subscheme of
$$W' := V( f_0, f_1, g ) = L \cup Q , \  L : =  V ( c_0, c_1) , \  Q 
:  = V (a, g).$$
We may assume W.L.O.G. that $L$ is the tangent cone of $W_1$, and $Q$ 
is the tangent cone of $W_2$

{\bf Step 5.} Note that $W' =  L \cup Q$ holds schematically, and we 
have a corresponding projective subscheme of codimension $2$ and 
degree $3$
of which  $ \sC$ has a subscheme. Hence there are no other components 
$W_i$ of codimension $2$: these would contribute to a higher degree 
of the subscheme $ \sC$. Hence, if there are other components $W_3, 
\dots$, they have codimension $3$.
We shall now show that these latter do not exist.

{\bf Step 6.} Since the tangent cone $L$ to $W_1 $ is smooth, then 
$W_1$ is smooth at the origin, and, by  a suitable local change of 
coordinates, we may assume that
$$ W_1 = V ( c_0, c_1).$$

Since  moreover $ F_0, F_1, G $ are in the ideal $ ( c_0, c_1)$, we 
obtain, after a suitable change of coordinates
$$ F_0 = c_0 a + c_1 \beta_1 , \  F_1 = c_1 ( a + \alpha ) +   c_0 
\beta_0, \ G = c_0 d - c_1 b,$$
where $\alpha, \beta_0, \beta_1$ have all order at least $2$ at the origin.
Moreover, we can assume (changing the coordinate $a$ and adding 
possibly a multiple of $G$ to $F_0$ and $F_1$  that

{\bf (6.1)  the variables $c_0$ and $b$ do not appear in  $ \beta_1$, 
and $\alpha$}

{\bf (6.1)  the variable  $c_1$ does not appear in  $ \beta_0$.}

{\bf Step 7.} Consider now $ W \setminus W_1$, and intersect with 
$c_0=0$: we obtain the algebraic set
$$ c_0 = b = \beta_1 = (a + \alpha) = 0,$$
which must have codimension $3$.

It follows that $ (a + \alpha) $ divides $\beta_1 $. We can therefore 
subtract a multiple of $F_1$ to $F_0$
and obtain that $c_0$ divides $F_0$.

Whence, we can finally assume that $ F_0 = c_0 a$.

{\bf Step 8.} Now, the components of $ V ( F_0 , G)$ are $W_1= L $, 
$Q = V (a, G)$ and $ V (c_0, b)$.

But $W \cap  V (c_0, b) = \{ c_0 = b = c_1 (a + \alpha) =0 \}$, which 
has  codimension $3$.
Whereas $ W \cap Q =  V ( a, G, c_1 \alpha + c_0 \beta_0)$; this 
component must have codimension $2$,
so it must be $Q$, and $F_1$ belongs to the ideal $ ( a, G)$. We can 
subtract a multiple of
$G$ to $F_1$ hence we may obtain that  $a$ divides $F_1$, i.e., that 
$a$ divides $\beta_0$ and $\alpha$.

{\bf Step 9.} We may now write
$$  F_1 = a ( c_1 (1 + A) + c_0 B),$$
hence, subtracting a multiple of $F_0$ to $F_1$, we may assume that $ 
B \equiv 0$.

Furthermore, multiplying by the unit $(1 + A)$ and its inverse the 
variables $c_1$ and $b$,
we finally get new coordinates where
$$  F_0 = c_0 a ,  F_1 =  c_1 a , G = c_0 d - c_1 b.$$

This is exactly what we wanted to show.

\end{proof}

\begin{theorem}\label{thatsallfolks}
$M_{8,4,0}^{ev}$ is connected and its irredundant  irreducible decomposition is
$M_{8,4,0}^{ev}=M_{\sF} \cup M_{\sE}$.
\end{theorem}

\begin{proof}
We already know that $M_{\sF} , M_{\sE}$ are irreducible unirational subsets
of $M_{8,4,0}^{ev}$: indeed each of them is  defined as the closure of a morphism
from an open set of an affine space into the moduli space.

We also know that $M_{\sF} $ is the closure of the open set  $M_{\sF} ^0$ of $M_{8,4,0}^{ev}$
consisting of surfaces with base point free canonical system: hence
$M_{\sF} $ is clearly an irreducible component of $M_{8,4,0}^{ev}$,
and it suffices to show that every surface not in  $M_{\sF} ^0$ lies either
in $M_{\sF} $ or in $ M_{\sE}$, and that these two subsets do intersect.
This will accomplish the proof.

Suppose then  that $S$ is an even surface with $K_S^2 = 8$, $p_g = 4$, $q = 0$ 
and that $|K_S|$ is not base point free.
As before, let $L$ be a half-canonical divisor, i.e., $K_S=2L$. 
Let $X$ be the canonical model of $S$.

Then $X$ is a small deformation of the cone $C_R$ over a smooth hyperplane section $H$ of $X$,
and the basis of the Kuranishi family $\sB$ of $C_R$ consists,  by \ref{ME}, \ref{MV}, \ref{secondlift} and \ref{tangentcone} of two irreducible components:
we want to show that the open set  $\sB'$ corresponding to the smoothings of $C_R$
remains connected, and that the points of one component of $\sB'$ correspond to
canonical models in $M_{\sF} $ and the points of the other to  canonical models in $ M_{\sE}$,
thus our statement shall follow.

We can translate everything back into the algebraic description of the extension rings  of $R=R(C,2Q)$ for a smooth 
$C\in |L|$ and $Q\in bs|L|$
(see Section \ref{sectionhyperplane}). 

The 
half-canonical ring
$R(X,L)$ is an  extension ring  of $R=R(C,2Q)$ hence  it  can be put in 
the form of (\ref{firstorderreduction}),
(we take now a small and general specialization $\epsilon \in \CC$) and, by Proposition \ref{secondlift}, 
$c_0a_5=c_1a_5=c_0d_7-c_1b_6=0$ , up to higher order terms.
Let us note that, up to a coordinate change in $R$ we can always assume that
the coefficient of $x_2^7$ in $D$ vanishes: just choose  $\zeta$ and 
$\eta$ with a common zero.

Now, if  $c_0=c_1=0$, the equations are in the extrasymmetric format of 
Proposition \ref{ME},
so $X$ is an element of $M_{\sE}$.
Else, $a_5=0$. Moreover, since we assumed the vanishing of the 
coefficient of $x_2^7$ in $D$, we can
decompose $\sB$ and $\sD$ as in Proposition \ref{MV} with 
$\sB_x=b_6x_1^6$ and $\sD_x=d_7x_1^7$.
Then $c_0d_7-c_1b_6=0$ gives $c_1x_1\sB_X=c_0\sD_x$.

 It follows that 
$R(X,L)$ is as described
in Proposition \ref{MV} with $l=c_1x_1$, and therefore $X \in M_{\sF}$.

 We see then directly that the base $\sB$ of the Kuranishi family of $C_R$ contains the subset   $\sB'' : = \{c_0a_5=c_1a_5=c_0d_7-c_1b_6=0\},$
hence  by Lemma \ref{tangentcone} it  equals $\sB$ and we have shown that $X$ belongs to  $M_{\sF} \cup M_{\sE}$.

Finally, the condition $M_{\sF} \cap M_{\sE} \neq 0$ was shown already in Proposition \ref{MV}.

\end{proof}

\end{document}